\documentclass[11pt,a4paper]{article}

\usepackage{array}
\newcolumntype{"}{@{\hskip\tabcolsep\vrule width 1.5pt\hskip\tabcolsep}}

\usepackage{graphics}
\usepackage{epsfig}
\usepackage{subfigure}
\usepackage{amssymb}
\usepackage{amsthm}
\usepackage{mathrsfs}
\usepackage{array}
\usepackage{hhline}
\usepackage{longtable}
\usepackage{amsmath}
\usepackage{float}
\usepackage{picinpar}
\usepackage{cite}
\usepackage{enumerate}
\usepackage{xcolor}
\usepackage{multirow}
\usepackage[mathlines]{lineno}
\modulolinenumbers[2]
\newcommand*\patchAmsMathEnvironmentForLineno[1]{%
\expandafter\let\csname old#1\expandafter\endcsname\csname #1\endcsname  \expandafter\let\csname oldend#1\expandafter\endcsname\csname end#1\endcsname  \renewenvironment{#1}%
{\linenomath\csname old#1\endcsname}%
{\csname oldend#1\endcsname\endlinenomath}}%
\newcommand*\patchBothAmsMathEnvironmentsForLineno[1]{%
\patchAmsMathEnvironmentForLineno{#1}%
\patchAmsMathEnvironmentForLineno{#1*}}%
\AtBeginDocument{%
\patchBothAmsMathEnvironmentsForLineno{equation}%
\patchBothAmsMathEnvironmentsForLineno{align}%
\patchBothAmsMathEnvironmentsForLineno{flalign}%
\patchBothAmsMathEnvironmentsForLineno{alignat}%
\patchBothAmsMathEnvironmentsForLineno{gather}%
\patchBothAmsMathEnvironmentsForLineno{multline}%
}

\def\N{\mathbb{N}}

\def\D{\Delta}
\def\rsq{\hspace*{\fill}$\blacksquare$\medskip}

\newtheorem{theorem}{Theorem}[section]
\newtheorem{lemma}[theorem]{Lemma}
\newtheorem{corollary}[theorem]{Corollary}

\newtheorem{conjecture}{Conjecture}[section]

\numberwithin{equation}{section}

\newtheoremstyle{example}
  {10pt}          
  {10pt}  
  {\rm}  
  {}
  {\bf}  
  {: }    
  { }    
  {}     
\theoremstyle{example}
\newtheorem{example}{Example}[section]

\def\N{\mathbb{N}}
\newtheorem{defi}{Definition}[section]

\newcommand{\T}[1]{\fontsize{#1}{#1}\selectfont}

\usepackage{graphicx}
\graphicspath{%
    {converted_graphics/}
    {/}
}

\def\ms{\medskip}

\def\nt{\noindent}

\begin{document}

\begin{center}
{\mathversion{bold}\Large \bf On Local Antimagic Chromatic Number of Spider Graphs}

\bigskip
{\large  Gee-Choon Lau$^{a,}$\footnote{Corresponding author.}, Wai-Chee Shiu{$^{b,c}$}, Chee-Xian Soo$^d$}\\

\medskip

\emph{{$^a$}Faculty of Computer \& Mathematical Sciences,}\\
\emph{Universiti Teknologi MARA (Segamat Campus),}\\
\emph{85000, Johor, Malaysia.}\\
\emph{geeclau@yahoo.com}\\

\medskip

\emph{{$^b$}Department of Mathematics, The Chinese University of Hong Kong,\\ Shatin, Hong Kong.}\\
\emph{{$^c$}College of Global Talents, Beijing Institute of Technology,\\ Zhuhai, China.}\\
\emph{wcshiu@associate.hkbu.edu.hk}\\

\medskip
\emph{{$^d$}School of Physical and Mathematical Sciences,\\Nanyang Technological University, 21 Nanyang Link,\\ 637731, Singapore.}\\
\emph{csoo002@e.ntu.edu.sg}\\

\end{center}


\begin{abstract}
An edge labeling of a connected graph $G = (V,E)$ is said to be local antimagic if it is a bijection $f : E \to \{1, . . . , |E|\}$ such that for any pair of adjacent vertices $x$ and $y$, $f^+(x) \ne f^+(y)$, where the induced vertex label $f^+(x) = \sum f(e)$, with $e$ ranging over all the edges incident to $x$. The local antimagic chromatic number of $G$, denoted by $\chi_{la}(G)$, is the minimum number of distinct induced vertex labels over all local antimagic labelings of $G$. In this paper, we first show that a $d$-leg spider graph has $d+1\le \chi_{la}\le d+2$. We then obtain many sufficient conditions such that both the values are attainable. Finally, we show that each 3-leg spider has  $\chi_{la} = 4$ if not all legs are of odd length. We conjecture that almost all $d$-leg spiders of size $q$ that satisfies $d(d+1) \le 2(2q-1)$ with each leg length at least 2 has $\chi_{la} = d+1$. \\

\noindent Keywords: Local antimagic labeling, Local antimagic chromatic number, Spiders

\noindent 2010 AMS Subject Classifications: 05C78; 05C69.
\end{abstract}

\section{Introduction}

A connected graph $G = (V, E)$ is said to be local antimagic if it admits a local antimagic (edge) labeling, i.e., a bijection $f : E \to \{1,\ldots ,|E|\}$ such that the induced vertex labeling  $f^+ : V \to \mathbb Z$ given by $f^+(x) = \sum f(e)$ (with $e$ ranging over all the edges incident to $x$) has the property that any two adjacent vertices have distinct induced vertex labels.  The number of distinct induced vertex labels under $f$ is denoted by $c(f)$, and is called the {\it color number} of $f$. Also, $f$ is call a {\it local antimagic $c(f)$-labeling} of $G$. The {\it local antimagic chromatic number} of $G$, denoted by $\chi_{la}(G)$, is $\min\{c(f) : f\mbox{ is a local antimagic labeling of } G\}$ (see~\cite{Arumugam-Wang, Arumugam, LNS-GC, LauNgShiu-DMGT, LSN, LauShiuNg-pendants}). For integers $a<b$, we let $[a,b]=\{a,a+1,a+2,\ldots,b\}$.

The following two results in~\cite{LSN} are needed.

\begin{lemma}\label{lem-pendant} Let $G$ be a graph of size $q$ containing $d$ pendants. Let $f$ be a local antimagic labeling of $G$ such that $f(e)=q$. If $e$ is not a pendant edge, then $c(f)\ge d+2$.\end{lemma}

\begin{theorem}\label{thm-pendant}  Let $G$ be a graph having $d$ pendants. If $G$ is not $K_2$, then $\chi_{la}(G)\ge d+1$ and the bound is sharp.\end{theorem}

For $1\le i\le t$, $a_i, n_i\ge 1$ and $d=\sum\limits^t_{i=1} n_i\ge 3$, a spider of $d$ legs, denoted $Sp(a_1^{[n_1]}, a_2^{[n_2]}, \ldots, a_t^{[n_t]})$ (or $Sp(y_1,y_2,\ldots,y_d)$ for convenience), is a tree formed by identifying an end-vertex of $n_i$ path(s) of length $a_i$. The vertex $u$ of degree $d$ is the core of the spider. Note that, $Sp(1^{[d]})$ is the star graph of $d$ pendant vertices with $\chi_{la}(Sp(1^{[d]})) = d+1$.
In this paper, we first show that $d+1\le \chi_{la}(Sp(y_1,y_2,\ldots,y_d))\le d+2$. We then obtain many sufficient conditions such that both the values are attainable. Finally, we show that $\chi_{la}(Sp(y_1,y_2,y_3)) = 4$ if not all $y_1,y_2,y_3$ are odd. We conjecture that each $d$-leg spider of size $q$ that satisfies $d(d+1) \le 2(2q-1)$ with each leg length at least 2 has $\chi_{la} = d+1$ except $Sp(2^{[n]},3^{[m]})$ for $(n,m)\in\{(4,0), (5,0), (6,0), (0,10), (1,8), (1,9), (2,7), (2,8), (3,5), (3,6), (4,4), (4,5), (5,3)\}$.    

\section{Spider Graphs}

In \cite{ LauShiuNg-pendants}, the authors gave a family of $d$-leg spiders to have $\chi_{la} > d+1$.

\begin{theorem}\label{thm-sp2} For $d\ge 3$, $$\chi_{la}(Sp(2^{[d]}))=\begin{cases}d+2 &\mbox{ if } d \ge 4 \\ d+1 & \mbox{ if }d=3. \end{cases}$$
\end{theorem}

\begin{theorem}\label{thm-maxdeg}
Let $G$ be a graph of size $q$ with $k\ge 1$ pendant vertices. Suppose $G$ has only one vertex of maximum degree $\D$ which is not adjacent to any pendant vertex and all other  vertices of $G$ has degree at most  $m<\D$. If $\D(\D+1) > m(2q-m+1)$, then $\chi_{la}(G)\ge k+2$.
\end{theorem}

\begin{proof}
Let $f$ be a local antimagic labeling of $G$.  If $q$ is assigned to a non-pendant edge, by Lemma~\ref{lem-pendant}, $c(f)\ge m+2$. Assume $q$ is assigned to a pendant edge that has a non-pendant end-vertex $x$. Now, all the $k$ pendant vertex labels of $G$ are distinct and at most $q$. Suppose $u$ is a vertex of degree $\D$, then $f^+(u)\ge \D(\D+1)/2$ and that $q+1\le f^+(x)\le m(2q-m+1)/2$. By the given hypothesis, $f^+(u)>f^+(x) > f^+(y)$ for every pendant vertex $y$.  Thus, $c(f)\ge k+2$. The theorem holds.
\end{proof}

\nt Note that if $G$ has at least two vertices of maximum degree $\D$, the condition $\D(\D+1) > \D(2q-\D+1)$ implies that $\D > q$, which is impossible. Also note that the conditions of Theorem~\ref{thm-maxdeg} is not necessary to have $\chi_{la}(G)\ge k+2$. A counterexample of non-tree graph is $\chi_{la}(K_3 \odot O_2)=9$, and counterexamples of trees are $Sp(2^{[n]}), n = 4,5,6$.

\begin{theorem}\label{thm-leglen1} A spider of $d \ge 3$ legs that has at least a leg of length $1$ has $\chi_{la} = d+1$. \end{theorem}

\begin{proof} Let $G = Sp(a_1,a_2,\ldots,a_r,1^{[t]})$ such that $r\ge 1$, $t\ge 1$, $a_i\ge 2$, and $d=r+t\ge 3$. Thus, $G$ is of size $q=\sum^r_{i=1}a_i + t$. Arrange all the paths of lengths $a_1, a_2, \ldots, a_r$ horizontally from left to right, and name the edges from left to right as $e_1, e_2, \ldots, e_{q-t}$. Let $E = \{e_j\,|\, 1\le j\le q-t\}$. Define $f : E \to [1,q-t]$ such that $f(e_j) = j/2$ for even $j$ and $f(e_j) = (q-t)-(j-1)/2$ for odd $j$. Finally, label the edges of the remaining $t$ path(s) of length 1 by $q-t+1$ to $q$ bijectively.

\ms\nt Identify all the $r+t$ right end-vertices of the paths as the core. We now have the graph $G$ with an induced edge labeling given by $f : E(G) \to [1,q]$. Observe that
\begin{enumerate}[(i)]
  \item each degree 2 vertex has induced label $j/2 + (q-t)-j/2 = q-t$ or $(q-t)-(j-1)/2 + (j+1)/2 = q-t+1$.
  \item each of the $t$ pendant vertices of paths length 1 has induced label $q-t+1, \ldots, q$ respectively.
  \item the pendant vertex of path length $a_1$ has induced label $q-t$ while all remaining $r-1$ pendant vertices have mutually distinct induced label at most $q-t-1$.
  \item the core has induced label larger than $q$.
\end{enumerate}

\nt Thus, $f$ is a local antimagic labeling with $c(f) = r+t+1=d+1$ and $\chi_{la}(G)\le d+1$. By Theorem~\ref{thm-pendant}, $\chi_{la}(G)\ge d+1$. The theorem holds.
\end{proof}

\begin{example}  In the following figure, we give the above defined labeling for $Sp(4,2,3,5,1)$.\\
\centerline{\begin{tabular}{m{7cm}p{6cm}}
\epsfig{file=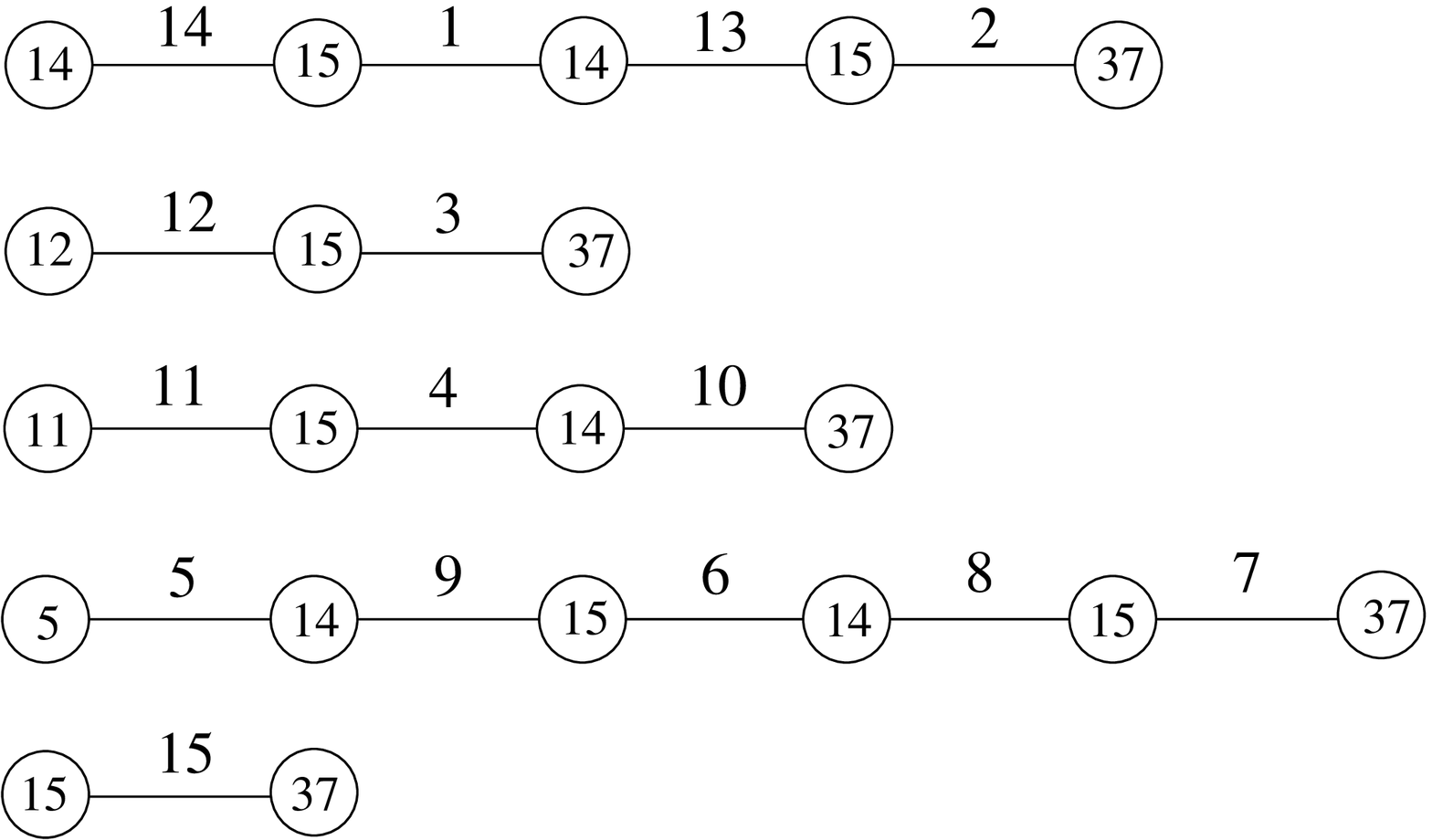,width=5cm} &  
$\begin{array}{*{5}{c}}
14 & 1 & 13 & 2\\
12 & 3\\
11 & 4 & 10\\
5 & 9 & 6 & 8 & 7\\
15\\
$(15,14,12,11,5)$
\end{array}$\\
\end{tabular}}

The right vertex of each path are identified as the core $u$. We use the right table to denote this labeling together with the induced vertex labels. We shall also use such kind of table to represent a labeling for a spider in the rest of the paper.
\rsq
\end{example}

\nt Call the labeling of the paths of length $a_i\ge 2, 1\le i\le r$ in Theorem~\ref{thm-leglen1} the {\it fundamental labeling} of spider.  By using this labeling, it is easy to get the following result:

\begin{corollary}\label{cor-spider}
Suppose $G$ is a spider with $d\ge 3$ legs, then $d+1\le \chi_{la}(G)\le d+2$.
\end{corollary}

\ms\nt In what follows, we assume every leg of spider is of length at least 2.

\begin{theorem}\label{thm-legnum}
Suppose $y_1, y_2,\ldots, y_d\ge 2$ and $d\ge 3$. If  $d(d+1)> 2(2q-1)$, then $\chi_{la}(Sp(y_1,y_2,\ldots,y_d)) = d+2$.
\end{theorem}

\begin{proof}
Let $g$ be a local antimagic labeling of $G=Sp(y_1,y_2,\ldots,y_d)$. Under the hypothesis, by Theorem~\ref{thm-maxdeg}, we know  $c(g)\ge d+2$. By Corollary~\ref{cor-spider}, the theorem holds.
\end{proof}

\begin{corollary} There are infinitely many spiders of $d\ge 3$ legs with $\chi_{la} = d+2$. \end{corollary}

\begin{theorem}\label{thm-len23} For $n\ge 0, m\ge 1$ and $n+m\ge 3$, $\chi_{la}(Sp(2^{[n]},3^{[m]}))=n+m+2$ if $(n,m)\not\in A$, where

\centerline{$\begin{aligned}A= & \{(0,i_0)\;|\;i_0\in[3,10]\}\cup\{(1,i_1)\;|\; i_1\in[2,9]\}\cup\{(2,i_2)\;|\; i_2\in[1,8]\}\cup\{(3,i_3)\;|\; i_3\in[1,6]\}\\& \cup\{(4,i_4)\;|\; i_4\in[1,5]\}\cup\{(5,1), (5,2), (5,3), (6,1)\}.\end{aligned}$}
\end{theorem}

\begin{proof} Let $G=Sp(2^{[n]},3^{[m]})$. Now, $q=2n+3m$ and $d=n+m$.  The inequality $(n+m)(n+m+1) \le 2(4n+6m-1)$ is equivalent to $(n+m-\frac{7}{2})^2\le 4m+\frac{41}{4}$.

\ms\nt If $m\ge 11$, then $(n+m-\frac{7}{2})^2\ge (m-\frac{7}{2})^2=m^2-7m+\frac{49}{4} > 4m+\frac{41}{4}$.

\ms\nt If $n\ge 7$, then $(n+m-\frac{7}{2})^2\ge m^2+7m+\frac{49}{4}>4m+\frac{41}{4}$.

\ms\nt So we have to deal with $1\le m\le 10$ and $0\le n\le 6$. It is routine to check that $(n+m)(n+m+1) \le 2(4n+6m-1)$ if and only if $(n,m)\in A$.  In other word, $(n,m)\not\in A$ if and only if $d(d+1) > 2(2q-1)$. By Theorem~\ref{thm-legnum}, we know that $\chi_{la}(G)=n+m+2$ if $(n,m)\not\in A$.
Thus, the theorem holds.
\end{proof}

Let us analyze the graph $G=Sp(2^{[n]},3^{[m]})$, where $(n,m)\in A$ and $A$ is defined in Theorem~\ref{thm-len23}.
By Appendix, it suffices to consider $(n,m)\in\{(0,10), (1,8), (1,9), (2,7), (2,8), (3,5), (3,6), (4,4), (4,5), (5,3)\}$.

Denote the legs of length 2 as $P_3^{(i)}=x_{i1}x_{i2}x_{i3}$, $1\le i\le n$ and those of length 3 as $P_4^{(j)}=y_{j1}y_{j2}y_{j3}y_{j4}$, $1\le j\le m$. Also the core $u=x_{i3}=y_{j4}$ for all $i,j$. Let $f$ be a local antimagic $(m+n+1)$-labeling. Note that, $f^+(u)\ge \frac{1}{2}(n+m)(n+m+1)\ge 4n+4m+4>2n+3m=q$.

By Lemma~\ref{lem-pendant}, $q$ is labeled at a pendant edge.
Suppose $q=2n+3m$ is labeled at a pendant edge of a leg of length 2 ($n\ge 1$). Without loss of generality, we may assume $f(x_{11}x_{12})=q$.
Since $f^+(x_{12})\ge q+1$ and $f^+(x_{i1})$ and $f^+(y_{j1})$ are at most $q$ for $1\le i\le n$ and $1\le j\le m$, $f^+(x_{13})=f^+(u)\le q$ which is a contradiction.
So $q$ must be labeled at a pendant edge of a leg of length 3.
By symmetric we may assume that $f(y_{11}y_{12})=q$. Then, $f^+(y_{11})=q$, $f^+(y_{12})\ge q+1$ and the other induced vertex colors are less than $q$. Thus $f^+(u)=f^+(y_{12})$. Let $f(y_{12}y_{13})=x$. Then $f^+(u)=f^+(y_{12})=q+x$. Then $q+x\ge \sum\limits_{i=1}^{m+n} i=\frac{1}{2}(m+n)(m+n+1)$.
So
 \begin{equation}\label{eq-cond1}
x\in[\textstyle\frac{1}{2}(n+m)(n+m+1)-(2n+3m),\ 2n+3m-1].\end{equation}

\begin{enumerate}[(A)]
\item Suppose $n=0$. Only the case $(n,m)=(0,10)$. From \eqref{eq-cond1} we have $x\in[25,29]$. Following table lists all the possible labels assigned at the edges incident to the core:
\[\begin{array}{c|c|l}
x & f^+(u) & \mbox{labels incident to the core}\\\hline
25 & 55 & 1,2,3,4,5,6,7,8,9,10\\
26 & 56 & 1,2,3,4,5,6,7,8,9,11\\
27 & 57 & 1,2,3,4,5,6,7,8,10,11\\
27 & 57 & 1,2,3,4,5,6,7,8,9,12\\
28 & 58 & 1,2,3,4,5,6,7,8,9,13\\
28 & 58 & 1,2,3,4,5,6,7,8,10,12\\
28 & 58 & 1,2,3,4,5,6,7,9,10,11\\
29 & 59 & 1,2,3,4,5,6,7,8,9,14\\
29 & 59 & 1,2,3,4,5,6,7,8,10,13\\
29 & 59 & 1,2,3,4,5,6,7,9,10,12\\
29 & 59 & 1,2,3,4,5,6,8,9,10,11
\end{array}\]
Note that integers in $[1,6]$ are labeled to edges incident to the core (we shall say that the labels are {\it incident to} the core). So, the edge with label 24 must be adjacent to an edge with label $s, 7\le s\le 29$. Thus the incident vertex of these two labels must have label $q+x$. This means $s=q+x-24 = x+6\ge 31$ which is impossible. Thus $\chi_{la}(Sp(3^{[10]}))=12$.

\item Suppose $n=1$. There are two cases $(n,m)=(1,8), (1,9)$.
When $(n,m)=(1,8)$.  From \eqref{eq-cond1} we have $x\in[19,25]$.
It is easy to check that labels in $[1,28-x]$ are incident to the core. Since there is only one leg of length 2,
there exists $y\in\{x-1, x-2\}$ is labeled at an edge of the leg of length 3. Since labels in $[1,28-x]$ are incident to the core, similar to Case~(A), we get that $q+x-y\ge q+1$. So, there is no local antimagic $10$-labeling for $Sp(2, 3^{[8]})$. Thus, $\chi_{la}(Sp(2, 3^{[8]}))=11$.

When $(n,m)=(1,9)$.  From \eqref{eq-cond1} we have $x\in[26,27]$. By a similar argument as above, we get that $\chi_{la}(Sp(2,3^{[9]}))=12$.

\item Suppose $n=2$. There are two cases $(n,m)=(2,7), (2,8)$.

When $(n,m)=(2,7)$. From \eqref{eq-cond1} we have $x\in[20,24]$. It is easy to check that labels in $[1,29-x]$ are incident to the core. Since there is only two legs of length 2,
there exists $y\in\{x-1, x-2, x-3\}$ is labeled at an edge of the leg of length 3. Since labels in $[1,29-x]$ are incident to the core, similar to Case~(A), we get that $q+x-y\ge q+1$. So, there is no local antimagic $10$-labeling for $Sp(2^{[2]}, 3^{[7]})$. Thus, $\chi_{la}(Sp(2^{[2]}, 3^{[7]}))=11$.

When $(n,m)=(2,8)$. From \eqref{eq-cond1} we have $x=27$. By a similar argument, we have $\chi_{la}(Sp(2^{[2]},3^{[8]}))=12$.

\item For the case $(n,m)\in\{(3,5), (3,6), (4,4), (4,5), (5,3)\}$. The proof is similar to the above cases. So we omit here.
\end{enumerate}

Hence we have,

\begin{theorem}
For $n\ge 0$, $m\ge 1$ and $n+m\ge 3$, $\chi_{la}(Sp(2^{[n]},3^{[m]}))=n+m+1$ if and only if $(n,m)\in B$, where

\centerline{$\begin{aligned}B= & \{(0,i_0)\;|\;i_0\in[3,9]\}\cup\{(1,i_1)\;|\; i_1\in[2,7]\}\cup\{(2,i_2)\;|\; i_2\in[1,6]\}\cup\{(3,i_3)\;|\; i_3\in[1,4]\}\\& \cup\{(4,i_4)\;|\; i_4\in[1,3]\}\cup\{(5,1), (5,2), (6,1)\}.\end{aligned}$}\end{theorem}

\begin{theorem}\label{thm-evenlegs} Let $y_1, \ldots, y_d\ge 2$ be evens such that $y_d = y_{d-2} + 2y_{d-3} + \cdots + (d-3)y_2 + (d-2)y_1$. Then $\chi_{la}(Sp(y_1,\ldots,y_d)) = d+1$, where $d\ge 2$.  \end{theorem}

\begin{proof} Let $G = Sp(y_1,y_2,\ldots,y_d)$ that satisfies the given conditions. Note that, the size of $G$ is $q = \sum\limits^d_{i=1} y_i$ and the condition can be rewritten as $\sum\limits_{i=1}^d (d-i-1)y_i=0$. Apply the fundamental labeling of spider to $G$. Identify all the right end-vertices (with label $\frac{1}{2}(y_1+\cdots +y_i)$) for path of length $y_i$. We now have the graph $G$ with an induced edge labeling $f: E(G)\to [1,q]$ such that
\begin{enumerate}[(i)]
  \item each degree 2 vertex has induced label $i/2 + q-i/2 = q$ or $q - (i-1)/2 + (i+1)/2 = q+1$. Moreover, the induced label of the second or the last second vertex of each path is $q+1$.
  \item the end-vertex incident to edge $e_1$ has induced label $q$ and the remaining $d-1$ end-vertices have mutually distinct induced labels at least $q/2 + 1$ and at most $q-1$.
  \item the core has induced label
\begin{align}
f^+(u) & = y_1/2 + (y_1+y_2)/2 + (y_1+y_2+y_3)/2 + \cdots + (y_1 + y_2 + \cdots + y_{d})/2 \nonumber\\
 & =  \frac{1}{2}\sum_{i=1}^d(d-i+1)y_i=\frac{1}{2}\sum_{i=1}^d(d-i-1)y_i+q \label{eq-evenlegs}\\ &= q.\nonumber
\end{align}
\end{enumerate}
\nt Thus, $f$ is a local antimagic labeling with $c(f) = d+1$ and $\chi_{la}(G)\le d+1$. Hence, the theorem holds by  Theorem~\ref{thm-pendant}.
\end{proof}

\newpage
\begin{example} Consider $Sp(2,4,6,4,20)$ which satisfies the hypothesis of Theorem~\ref{thm-evenlegs}.
According to the fundamental labeling we have:\\
36, 1;\\ 
35, 2, 34, 3;\\
33, 4, 32, 5 ,31 ,6;\\
30, 7, 29, 8;\\
28, 9, 27, 10, 26, 11, 25, 12, 24, 13, 23, 14, 22, 15, 21, 16, 20, 17, 19, 18.\\
$(37,36,36,33,30)$ \rsq 
\end{example}

\begin{corollary} \label{cor-evenlegs} Let $y_1, \ldots,y_d\ge 2$ be evens. If $\sum\limits_{i=1}^{d-1} (d-i)y_i=\sum\limits_{i=k+1}^{d} y_i$, where $0\le k\le d-1$, then $\chi_{la}(Sp(y_1,\ldots,y_d))=d+1$.
\end{corollary}
\begin{proof} Keep the notation in the proof of Theorem~\ref{thm-evenlegs}.
From the labeling $f$, we have known that each pendant vertex has induced label $q-\frac{1}{2}\sum\limits_{i=1}^{k} y_i$, $0\le k\le d-1$. Note that when $k=0$ the empty summation is treated as zero. Thus, $\chi_{la}(Sp(y_1, \dots, y_d)=d+1$ if $f^+(u)=q-\frac{1}{2}\sum\limits_{i=1}^{k} y_i$ for some $0\le k\le d-1$.
From \eqref{eq-evenlegs}, we have $\frac{1}{2}\sum\limits_{i=1}^{d}(d-i-1)y_i + q= q-\frac{1}{2}\sum\limits_{i=1}^{k} y_i$.
It is equivalent to $\sum\limits_{i=1}^{d}(d-i)y_i =\sum\limits_{i=1}^{d}y_i -\sum\limits_{i=1}^{k} y_i$ or  $\sum\limits_{i=1}^{d-1}(d-i)y_i=\sum\limits_{i=k+1}^{d} y_i$.
\end{proof}
\begin{corollary}\label{cor-even+1odd}
Let $y_1, \ldots, y_d\ge 2$ be evens. If $\sum\limits_{i=1}^{d-1} (d-i)y_i=\sum\limits_{i=k+1}^{d} y_i$, where $1\le k\le d-1$, then $\chi_{la}(Sp(y_1,\ldots,y_d+1))=d+1$.
\end{corollary}

\begin{example}
Consider $y_1=4$, $y_2=2$, $y_3=4$, $y_4=2$.
\begin{enumerate}[1.]
\item  By Theorem~\ref{thm-evenlegs} or Corollary~\ref{cor-evenlegs} with $k=0$, we get $y_5=20$ and so $\chi_{la}(Sp(4,2,4,2,20))=6$.
According to the fundamental labeling we get the following labeling of each leg:\\
32, 1, 31, 2;\\
30, 3;\\
29, 4, 28, 5;\\
27, 6;\\
26, 7, 25, 8, 24, 9, 23, 10, 22, 11, 21, 12, 20, 13, 19, 14, 18, 15, 17, 16.\\
$(33, 32, 30, 29, 27, 26)$ 

\item By Corollary~\ref{cor-evenlegs} with $k=1$, we get $y_5=24$ and so $\chi_{la}(Sp(4,2,4,2,24))=6$. According to the fundamental labeling we get the following labeling of each leg:\\
36, 1, 35, 2, ;\\
34, 3;\\
33, 4, 32, 5;\\
31, 6;\\
30, 7, 29, 8, 28, 9, 27, 10, 26, 11, 25, 12, 24, 13, 23, 14, 22, 15, 21, 16, 20, 17, 19, 18.\\
$(37, 36, 34, 33, 31, 30)$  

\item By Corollary~\ref{cor-even+1odd}, from Case~2, $\chi_{la}(Sp(4,2,4,2,25))=6$.
According to the fundamental labeling we get the following labeling of each leg:\\
37, 1, 36, 2, ;\\
35, 3;\\
34, 4, 33, 5;\\
32, 6;\\
31, 7, 30, 8, 29, 9, 28, 10, 27, 11, 26, 12, 25, 13, 24, 14, 23, 15, 22, 16, 21, 17, 20, 18, 19.\\
$(38, 37, 35, 34, 32, 31)$  
\end{enumerate} 
\noindent By Corollary~\ref{cor-even+1odd} with $k=4$, we also get $\chi_{la}(Sp(4,2,4,2,33))=6$.\rsq
\end{example}

\begin{example} Let $y_1=4$ $y_2=6$. Using $k=1$, $2y_1+y_2=14=y_2+y_3$ gives $y_3=8$. So, we get $Sp(4,6,8)$ with the following labeling.\\
18, 1, 17, 2;\\
16, 3, 15, 4, 14, 5;\\
13, 6, 12, 7, 11, 8, 10, 9\\
$(19,18,16,13)$   

\nt By Corollary~\ref{cor-even+1odd}, we also obtain $\chi_{la}(Sp(4,6,13))=\chi_{la}(Sp(4,6,19))=4$.
\rsq
\end{example}

\begin{corollary}\label{cor-2even} If $a,b\ge 2$ are even, then $\chi_{la}(Sp(a,b,a)) =  \chi_{la}(Sp(a,b,2a)) = \chi_{la}Sp(a,b,2a+b)) = \chi_{la}(Sp(a,b,a+1))= \chi_{la}(Sp(a,b,2a+1)) = \chi_{la}Sp(a,b,2a+b+1)) = 4$. \end{corollary}

\section{Spiders with 3 legs}

Before considering some labelings of spider of 3 legs, we show some useful labelings for a path first.

\begin{lemma}\label{lem-circ-perm}
Suppose $N\ge 2$ and
\[r\in \begin{cases}
\{1\}, & N=2;\\
\{2j\;|\; 1\le j\le N/2-1\}\cup\{1,N-1\}, & \mbox{even } N\ge 4;\\
\{2j-1\;|\; 1\le j\le (N-1)/2\}\cup\{N-1\}, & \mbox{odd } N.
\end{cases}\]
There is a circular permutation $(a_i)_{i=1}^N$ of $[1,N]$ such that the sum of two consecutive terms $a_i+a_{i+1}\in \{N, N+1, N+2\}$ for $1\le i\le N$ and $|a_N-a_1|=r$, where the indices of $a_i$'s are taken modulo $N$.
\end{lemma}

\begin{proof} We separate into 4 cases.
\begin{enumerate}[(a)]
\item $N=4k$, $k\ge 1$.
\[\begin{array}{|c*{8}{|c}|}
i & 1 & 2 & 3 & 4 & \cdots & 2k-2 & 2k-1 & 2k \\\hline\hline
a_i & 2k & 2k+2 & 2k-2 & 2k+4 & \cdots & 4k-2 & 2 & 4k\\\hline
a_{2k+i} & 1 & 4k-1 & 3 & 4k-3 & \cdots & 2k+3 & 2k-1 & 2k+1\\\hline
\end{array}\]
The set of difference of two consecutive terms of the obtained sequence is 
$\{|a_i-a_{i+1}|\;:\; 1\le i\le 4k\}=\{2j\;|\; 1\le j\le 2k-1\}\cup\{1,4k-1\}$. When we shift suitably the subscripts of $a_i$ we will obtain that $|a_N-a_1|=r$. For the remaining cases, we will not point out this arrangement again.

\item $N=4k+1$, $k\ge 1$.
\[\begin{array}{|c*{9}{|c}|}
i & 1 & 2 & 3 & 4 & \cdots & 2k-2 & 2k-1 & 2k & 2k+1\\\hline\hline
a_i & 2k+2 & 2k-1 & 2k+4 & 2k-3 & \cdots & 3 & 4k & 1 & \varnothing\\\hline
a_{2k+i} & 4k+1 & 2 & 4k-1 & 4 & \cdots & 2k-2 & 2k+3 & 2k & 2k+1\\\hline
\end{array}\]
The set of  difference of two consecutive terms of the obtained sequence is
$\{|a_i-a_{i+1}|\;:\; 1\le i\le 4k+1\}=\{2j-1\;|\; 1\le j\le 2k\}\cup\{4k\}$.

\item $N=4k+2$, $k\ge 0$.
\[\begin{array}{|c*{8}{|c}|}
i & 1 & 2 & 3 & 4 & \cdots & 2k-1 & 2k & 2k+1\\\hline\hline
a_i & 2k+1 & 2k+3 & 2k-1 & 2k+5 & \cdots & 3 & 4k+1 & 1\\\hline
a_{2k+i+1} & 4k+2 & 2 & 4k & 4 & \cdots & 2k+4 & 2k & 2k+2\\\hline
\end{array}\]
The set of difference of two consecutive terms of the obtained sequence is
$\{|a_i-a_{i+1}|\;:\; 1\le i\le 4k+2\}=\{2j\;|\; 1\le j\le 2k\}\cup\{1,4k+1\}$, where $k\ge 1$.
When $k=0$, this set is $\{1\}$.
\item $N=4k+3$, $k\ge 0$.
\[\begin{array}{|c*{9}{|c}|}
i & 1 & 2 & 3 & 4 & \cdots & 2k-1 & 2k & 2k+1 & 2k+2\\\hline\hline
a_i & 2k+1 & 2k+4 & 2k-1 & 2k+6 & \cdots & 3 & 4k+2 & 1 & \varnothing\\\hline
a_{2k+i+1} & 4k+3 & 2 & 4k+1 & 4 & \cdots & 2k+5 & 2k & 2k+3 & 2k+2\\\hline
\end{array}\]
The set of difference of two consecutive terms of the obtained sequence is
$\{|a_i-a_{i+1}|\;:\; 1\le i\le 4k+3\}=\{2j-1\;|\; 1\le j\le 2k+1\}\cup\{4k+2\}$.
\end{enumerate}

Suppose $|a_j-a_{j+1}|=r$. We may choose the circular permutation starting at $a_{j+1}$ and ending at $a_j$. This is a required sequence.
\end{proof}

Actually, these labelings induce local antimagic $3$-labelings of cycles accordingly.

\ms
Adding each term of the circular permutation above by $a\in\N$, we have

\begin{corollary}\label{cor-circ-perm}
Suppose $N\ge 2$, $a\in \N$ and
\[r\in \begin{cases}
\{1\}, & N=2;\\
([2,N-2]\cap \mathbb E)\cup\{1,N-1\}, & \mbox{even } N\ge 4;\\
([1,N-2]\cap\mathbb O)\cup\{N-1\}, & \mbox{odd } N,
\end{cases}\]
where $\mathbb E$ and $\mathbb O$ are the set of even and odd integers, respectively.
All integers in $[a+1,a+N]$ can be arranged as a sequence $(a_i)_{i=1}^N$ of length $N$ such that the sum of two consecutive terms $a_i+a_{i+1}\in \{2a+N, 2a+N+1, 2a+N+2\}$ for $1\le i\le N-1$ and $|a_N-a_1|=r$.
\end{corollary}

In this section, we consider $Sp(n_1,n_2,n_3)$, spider graphs with $3$ legs. Let its three legs be $P_{n_1+1}=x_{1}\cdots x_{n_1}x_{n_1+1}$, $P_{n_2+1}= y_{1}\cdots y_{n_2}y_{n_2+1}$ and $P_{n_3+1}=z_{1}\cdots z_{n_3}z_{n_3+1}$, where $x_{n_1+1}=y_{n_2+1}=z_{n_3+1}=u$.

Let $P=x_1\cdots x_n$ and $Q=y_1\cdots y_m$, $n,m\ge 2$. The graph $PQ$ is the path obtained from $P$ and $Q$ by identifying $x_n$ with $y_1$.

\begin{lemma}\label{lem-path-4a} Let $P_{n+1}=x_1\cdots x_{n+1}$ be a path of length $n\ge 2$. Let $a\notin\{0,1\}$. There is an edge labeling $f: E(P_{n+1})\to [a,a+n-1]$ such that $f^+$ is a coloring of $P_{n+1}$ and $f^+(x_1)=a+n-2$, $f^+(x_{n+1})=a+n-1$ and $f^+(x_i)\in\{2a+n-1, 2a+n-2, 2a+n-3\}$ for $2\le i\le n$. Moreover, $f^+(x_2)=2a+n-1=f^+(x_n)$.
\end{lemma}
\begin{proof}
Suppose $n=2k$ for some $k\ge 1$. Label the edges by
\[\begin{array}{c|*{9}{|c}|}
i & 1 &  2 & 3 & 4 & \cdots & 2k-3 & 2k-2& 2k-1 & 2k\\\hline\hline
x_ix_{i+1} & a+2k-2 & a+1 & a+2k-4 & a+3& \cdots & a+2 & a+2k-3& a& a+2k-1\\\hline\end{array}\]
Precisely,
\[f(x_ix_{i+1})=
\begin{cases}
a+2k-i-1 & \mbox{ if $i$ is odd};
\\
a+i-1 & \mbox{ if $i$ is even}.
\end{cases}\]
It is easy to see that $f^+(x_1)=a+2k-2$,  $f^+(x_{2k+1})=a+2k-1$,  $f^+(x_i)=2a+2k-1$ for even $i$, $f^+(x_i)=2a+2k-3$ for odd $i\in[3, 2k-1]$. Clearly $f^+$ is a $4$-coloring.

Suppose $n=4k+1$ for some $k\ge 1$.
Label the edges by
\[\begin{array}{c|*{7}{|c}||c||}
i & 1 &  2 & 3 & 4 & \cdots & 2k-1 & 2k & 2k+1\\ \hline\hline
x_ix_{i+1} & a+4k-1 & a+1 & a+4k-3 & a+3& \cdots & a+2k+1 & a+2k-1 & a+2k\\\hline\end{array}\]
\[\begin{array}{c||c|*{9}{|c}|}
i & 2k+1 & 2k+2 & 2k+3 & \cdots  & 4k-2 & 4k-1 & 4k& x_{4k+1}x_{4k+2}\\\hline\hline
x_ix_{i+1} &  a+2k & a+2k-2 & a+2k+2 & \cdots & a+2 & a+4k-2& a& a+4k\\\hline\end{array}\]
Precisely,
\[f(x_ix_{i+1})=
\begin{cases}
a+4k-i & \mbox{ if $i$ is odd and $1\le i\le 2k-1$};\\
a+2k & \mbox{ if }i=2k+1;\\
a+i-1 & \mbox{ if $i$ is odd and $2k+3\le i\le 4k+1$};\\
a+i-1 & \mbox{ if $i$ is even and $2 \le i\le 2k$};\\
a+4k-i & \mbox{ if $i$ is even and $2k+2 \le i\le 4k$}.
\end{cases}\]
It is easy to see that $f^+(x_1)=a+4k-1$,  $f^+(x_{4k+2})=a+4k$,  $f^+(x_i)=2a+4k$ for even $i\in[2, 2k]$, $f^+(x_i)=2a+4k-2$ for even $i\in[2k+2, 4k]$, $f^+(x_i)=2a+4k-2$ for odd $i\in[3, 2k-1]$, $f^+(x_{2k+1})=2a+4k-1$, $f^+(x_i)=2a+4k$ for odd $i\in[2k+3, 4k+1]$. Clearly $f^+$ is a $5$-coloring.

Suppose $n=4k+3$ for some $k\ge 0$.
Label the edges by
\[\begin{array}{c|*{8}{|c}||c||}
i & 1 &  2 & 3 & 4 & \cdots & 2k-1 & 2k & 2k+1 & 2k+2\\ \hline\hline
x_ix_{i+1} & a+4k+1 & a+1 & a+4k-1 & a+3& \cdots & a+2k+3 & a+2k-1 & a+2k+1 & a+2k \\\hline\end{array}\]
\[\begin{array}{c||c|*{9}{|c}|}
i & 2k+2 & 2k+3 & 2k+4 & \cdots & 4k & 4k+1& 4k+2& 4k+3\\\hline\hline
x_ix_{i+1} & a+2k & a+2k+2 &a+2k-2 & \cdots & a+2 & a+4k& a& a+4k+2\\\hline\end{array}\]
Precisely,
\[f(x_ix_{i+1})=
\begin{cases}
a+4k+2-i & \mbox{ if $i$ is odd and $1\le i\le 2k+1$};\\
a+i-1 & \mbox{ if $i$ is odd and $2k+3\le i\le 4k+3$};\\
a+i-1 & \mbox{ if $i$ is even and $2 \le i\le 2k$};\\
a+4k+2-i & \mbox{ if $i$ is even and $2k+2 \le i\le 4k+2$}.
\end{cases}\]
It is easy to see that $f^+(x_1)=a+4k+1$,  $f^+(x_{4k+4})=a+4k+2$,  $f^+(x_i)=2a+4k+2$ for even $i\in[2, 2k]$, $f^+(x_{2k+2})=2a+4k+1$, $f^+(x_i)=2a+4k$ for even $i\in[2k+4, 4k+2]$, $f^+(x_i)=2a+4k$ for odd $i\in[3, 2k+1]$, $f^+(x_i)=2a+4k+2$ for odd $i\in[2k+3, 4k+3]$. Clearly $f^+$ is a $5$-coloring.

The last statement is easy to check.
\end{proof}

\begin{theorem}\label{thm-2el} $\chi_{la}(Sp(2,2+2m,2+2m+k))=4$ for $m\ge 0$ and $k\ge 2$.
\end{theorem}
\begin{proof} Let $f$ be the required labeling. Now  $q=6+4m+k$.
Let $P^{(1)}_{3}=x_{1}x_2 x_{3}$, $P^{(2)}_{3}= y_{1}y_2 y_{3}$ and $P^{(3)}_{3}=z_{1}z_2 z_{3}$, $Q^{(1)}_{2m+1}=v_1\cdots v_{2m}v_{2m+1}$,
$Q^{(2)}_{2m+1}=w_1\cdots w_{2m}w_{2m+1}$, and $R_{k+1}=u_1\cdots u_ku_{k+1}$.

We label $P^{(j)}_{3}$ by $f(x_1x_2)=q$, $f(x_2x_3)=1$, $f(y_1y_2)=q-1$, $f(y_2y_3)=2$, $f(z_1z_2)=q-2$, $f(z_2z_3)=3$ and $Q^{(l)}_{2m+1}$, $1\le l\le 2$, by
\[\begin{array}{c|*{7}{|c}|}
i & 1 & 2 & 3 & 4 & \cdots & 2m-1 & 2m \\\hline\hline
v_iv_{i+1} & q-3n & 3n+1 & q-3n-2 & 3n+3 & \cdots & q-3n-2m+2 & 3n+2m-1\\ \hline
w_iw_{i+1} & q-3n-1 & 3n+2 & q-3n-3 & 3n+4 &\cdots & q-3n-2m+1 & 3n+2m \\\hline
\end{array}\]
The vertex labels of $v_{1}\cdots v_{2m}$ and $w_{1}\cdots w_{2m}$ are $q+1$ and $q-1$ alternatively.

Let $P_3=P^{(1)}_3$, $P_{3+2m}=P^{(2)}_3Q^{(1)}_{2m+1}$ and $P_{3+2m+k}=P^{(3)}_3Q^{(2)}_{2m+1}R_{k+1}$.
By Lemma~\ref{lem-path-4a}, we may label $R_{k+1}$ by $[4+2m, 3+2m+k]$ such that $f^+(u_{k+1})=3+2m+k$, $f^+(u_1)=2+2m+k$ and $f^+(u_2)=7+4m+k=f^+(u_k)$. Let the combined labeling still be denoted by $f$. Then $f^+(z_{3})=f^+(y_{3})=q-1$,
$f^+(w_{2m+1})=(3+2m)+(2+2m+k)=5+4m+k$ and $f^+(u)=1+(2+2m)+(3+2m+k)=6+4m+k$.
\end{proof}

\begin{theorem}\label{thm-eol} Suppose $n\ge 1$, $m\ge 0$ and $l\ge m+2$. $\chi_{la}(Sp(2n,2n+2m+1,l))=4$.
\end{theorem}

\begin{proof}
Here the number of edges is $q=4n+2m+l+1$.
Let $P^{(1)}_{2n+1}=x_{1}\cdots x_{2n}x_{2n+1}$, $P^{(2)}_{2n+1}= y_{1}\cdots y_{2n+1}$ and $Q_{2m+2}=v_1\cdots v_{2m}v_{2m+2}$. Also, let the last leg of the spider be $P_{l+1}=z_1\cdots z_{l+1}$. Let $f$ be a required labeling. Firstly, let $f(z_1z_2)=q$, $f(z_2z_3)=x$ and $f(z_{l}z_{l+1})=y$.

We label $P^{(j)}_{2n+1}$ and $Q_{2m+2}$, $1\le j\le 2$, by
\[\begin{array}{c|*{7}{|c}|}
i & 1 & 2 & 3 & 4 & \cdots & 2n-1 & 2n\\\hline\hline
x_ix_{i+1} & q-1 & 1& q-3 & 3 & \cdots & q-2n+1 & 2n-1\\\hline
y_iy_{i+1} & q-2 & 2 & q-4 & 4 & \cdots & q-2n & 2n\\ \hline
\end{array}\]
\[\T{10}\begin{array}{c|*{9}{|c}|}
i & 1 & 2 & 3 & 4 & \cdots &2m-2 & 2m-1 & 2m & 2m+1\\\hline\hline
v_iv_{i+1} & q-2n-1 & 2n+1 & q-2n-2 & 2n+2 & \cdots &2n+m-1 & q-2n-m & 2n+m & q-2n-m-1\\ \hline
\end{array}\]

Following we shall label all $l-1$ unlabeled edges of the path $P_{l+1}=z_1\cdots z_{l+1}$ by integers in\break $[2n+m+1, 2n+m+l-1]=[a+1, a+N]$, where $a=2n+m$ and $N=l-1$.

\begin{enumerate}[(a)]
\item Suppose $l\equiv m\pmod 2$ and $l\ge m+4$.
Let the first two legs of the spider be $P_{2n+1}=P^{(2)}_{2n+1}$ and $P_{2n+2m+2}=P^{(1)}_{2n+1}Q_{2m+2}$.
We require $q+x=f^+(z_2)=f^+(u)=f(z_{l}z_{l+1})+(2n)+(q-2n-m-1)=y+q-m-1$ and $f^+(z_i)\in\{q-2, q-1, q\}$ for $3\le i\le l$. It is equivalent to $x+m+1=y$.

Let $r=m+1$. Then $N=l-1>r$ and $r\equiv N\pmod 2$.
By Corollary~\ref{cor-circ-perm} with this $r$, $N$ and $a$, we can label $z_2\cdots z_{l+1}$ satisfying the above requirement. Note that $2a+N=4n+2m+l-1$.
Hence $f$ is a local antimagic $4$-labeling for $Sp(2n,2n+2m+1,l)$.

\item Suppose $l\not\equiv m\pmod{2}$ and $l\ge m+4$. Let the first two legs of the spider be $P_{2n+1}=P^{(1)}_{2n+1}$ and $P_{2n+2m+2}=P^{(2)}_{2n+1}Q_{2m+2}$. Similarly, we require $x+m+2=y$.
Let $r=m+2$. Now $N> r$ and $N\equiv r\pmod 2$. By Corollary~\ref{cor-circ-perm}, we can label $z_2\cdots z_{l+1}$ satisfying the above requirement. Hence $f$ is a local antimagic $4$-labeling for $Sp(2n,2n+2m+1,l)$.

\item Suppose $l=m+3$. We use the same labeling as the case (a). We require that $x+m+1=y$. Let $r=m+1$. Then $N=m+2$, i.e., $r=N-1$. By Corollary~\ref{cor-circ-perm}, we have a local antimagic $4$-labeling for $Sp(2n,2n+2m+1,m+3)$.

\item Suppose $l=m+2$.
\begin{enumerate}[(i)]
\item Suppose $m\ge 1$. We use the labeling defined in the case (a), but swap the labels of the edges $v_{2m-1}v_{2m}$ and $v_{2m+1}v_{2m+2}$. Now, we require that $x+m=y$. Let $r=m$. Then $r=N-1$. By Corollary~\ref{cor-circ-perm}, we have a local antimagic $4$-labeling for $Sp(2n,2n+2m+1,m+2)$.

\item Suppose $m=0$. If $n\ge 2$, then $Sp(2n, 2n+1, 2)=Sp(2, 2+2(n-1)+1, 2n)$. Since $2n\ge n+2$, by the cases~(a), (b) or (c) we get the result. If $n=1$, then $Sp(2, 3, 2)=Sp(2, 2, 3)$. A required labeling is 7,2;\ 6,3;\ 5,4,1 with induced colors 9,7,6,5.
\end{enumerate}
\end{enumerate}\vskip-1cm
\end{proof}

\begin{theorem}\label{thm-oel} Suppose $n\ge 1$, $m\ge 0$ and $l\ge 2$. $\chi_{la}(Sp(2n+1,2n+2m,l))=4$.
\end{theorem}

\begin{proof}When $m=0$. $Sp(2n+1, 2n+2m, l)=Sp(2n,2n+1, l)$. By Theorem~\ref{thm-eol} we get $\chi_{la}(Sp(2n+1, 2n+2m, l))=4$. Following we assume $m\ge 1$.

Suppose $l=2$. $Sp(2n+1, 2n+2m, 2)=Sp(2,2(n-1)+3, 2n+2m)$. Clearly, $2n+2m>(n-1)+3$. By Theorem~\ref{thm-eol} (with $n=1$) we have $\chi_{la}(Sp(2,2(n-1)+3, 2n+2m))=4$.

So we assume $l\ge 3$.
\begin{enumerate}[A.]
\item Suppose $l\ge m$.
Let $P^{(1)}_{2n+2}=x_{1}\cdots x_{2n}x_{2n+2}$, $P^{(2)}_{2n+2}= y_{1}\cdots y_{2n+2}$ and $Q_{2m}=v_1\cdots v_{2m}$. Also, let the last leg of the spider be $P_{l+1}=z_1\cdots z_{l+1}$. Now the number of edges is $q=4n+2m+l+1$.

Let $f$ be a required labeling. Firstly, let
$f(z_1z_2)=q$, $f(z_2z_3)=x$ and $f(z_{l}z_{l+1})=y$.

We label $P^{(j)}_{2n+2}$ and $Q_{2m}$, $1\le j\le 2$, by
\[\begin{array}{c|*{8}{|c}|}
i & 1 & 2 & 3 & 4 & \cdots & 2n-1 & 2n & 2n+1\\\hline\hline
x_ix_{i+1} & q-1 & 1& q-3 & 3 & \cdots & q-2n+1 & 2n-1 & q-2n-1\\\hline
y_iy_{i+1} & q-2 & 2 & q-4 & 4 & \cdots & q-2n & 2n & q-2n-2\\ \hline
\end{array}\]
\[\begin{array}{c|*{7}{|c}|}
i & 1 & 2 & 3 & \cdots &2m-3 & 2m-2 & 2m-1 \\\hline\hline
v_iv_{i+1} & 2n+1 & q-2n-3 & 2n+2 & \cdots &2n+m-1 & q-2n-m-1 & 2n+m \\ \hline
\end{array}\]

Similar to the proof of Theorem~\ref{thm-eol}, we shall label all $l-1$ unlabeled edges of the path $P_{l+1}=z_1\cdots z_{l+1}$ by integers in $[a+1, a+N]$, where $a=2n+m$ and $N=l-1$.

We assume $m\ge 3$ first.
\begin{enumerate}[(a)]
\item Suppose $l\equiv m\pmod 2$ and $l> m$.
Let the first two legs of the spider be $P_{2n+2}=P^{(1)}_{2n+2}$ and $P_{2n+2m+1}=P^{(2)}_{2n+2}Q_{2m}$.
Similar to the proof of Theorem~\ref{thm-eol} we require $q+x=f^+(z_2)=f^+(u)=f(z_{l}z_{l+1})+(q-2n-1)+(2n+m)=y+q+m-1$. It is equivalent to $x=y+m-1$.
Let $r=m-1$ and $N=l-1$.
By Corollary~\ref{cor-circ-perm}, we have the result.

\item Suppose $l\not\equiv m\pmod 2$. Let the first two legs of the spider be $P_{2n+2}=P^{(2)}_{2n+2}$ and $P_{2n+2m+1}=P^{(1)}_{2n+2}Q_{2m}$. Similarly, we require $x=y+m-2$. Let $r=m-2$ and $N=l-1$.
By Corollary~\ref{cor-circ-perm}, we have the result.
\item Suppose $l=m$. We use the labeling defined in (b). Now $r=m-2=N-1$. By Corollary~\ref{cor-circ-perm}, we have the result.
\end{enumerate}

When $m=1$. Use the labeling of case~(b). We require $x=y-1$. That is, $r=1$ and $N=l-1\ge 2$. By Corollary~\ref{cor-circ-perm}, $f$ is a local antimagic $4$-labeling for $Sp(2n+1,2n+2m,l)$.

When $m=2$. Use the labeling of case~(a). We require $x=y+1$. By Corollary~\ref{cor-circ-perm}, $f$ is a local antimagic $4$-labeling for $Sp(2n+1,2n+2m,l)$.

\item Suppose $l<m$.

\begin{enumerate}[(a)]
\item Suppose $l=2h\le 2n$. Then $Sp(2n+1,2n+2m,l)=Sp(2h, 2n+1, 2n+2m)=Sp(2h, 2h+2(n-h)+1, 2n+2m)$. Clearly, $2n+2m\ge (n-h)+3$. By Theorem~\ref{thm-eol}, $\chi_{la}(Sp(2n+1,2n+2m,l))=4$.

\item Suppose $l=2n+2s$, where $s\ge 1$. $Sp(2n+1,2n+2m,l)=Sp(2n+1, 2n+2m, 2n+2s)=Sp(2n+1, 2n+2s, 2n+2m)$. Clearly, $2n+2m\ge\max\{3, s\}$.  By the Case~A, $\chi_{la}(Sp(2n+1,2n+2m,l))=4$.
\item Suppose $l=2h+1$. Since $m\ge 4$, we let $m= 2h+k, k\ge 2$ to get $Sp(2n+1,2h+1,2n+4h+2k)$. Now, $q=4n+6h+2k+2$.
Let $P^{(1)}_{2n+2}=x_{1}\cdots x_{2n+1}x_{2n+2}$, $P^{(2)}_{2n+2}= y_{1}\cdots y_{2n+1}y_{2n+2}$, $Q_{4h+1}=z_1\cdots z_{4h+1}$, $S_{2k}=w_1\cdots w_{2k}$ and $R_{2h+2}=v_1\cdots v_{2h+1}v_{2h+2}$. Let $f$ be a required labeling. Firstly, let $f(v_1v_2)=q$, $f(v_2v_3)=x$ and $f(v_{2h+1}z_{2h+2})=y$.

We label $P^{(j)}_{2n+2}$ by integers in $[1,2n]\cup [q-2n-2,q-1]$:
\[\begin{array}{c|*{8}{|c}|}
i & 1 & 2 & 3 & 4 & \cdots & 2n-1 & 2n & 2n+1\\\hline\hline
x_ix_{i+1} & q-1 & 1& q-3 & 3 & \cdots & q-2n+1 & 2n-1 & q-2n-1\\\hline
y_iy_{i+1} & q-2 & 2 & q-4 & 4 & \cdots & q-2n & 2n & q-2n-2 \\ \hline
\end{array}\]

We label $Q_{4h+1}$ by in integers in $[2n+1,2n+2h]\cup [2n+4h+2k,2n+6h+2k-1=q-2n-3]$:
\[\T{9}\begin{array}{c|*{9}{|c}|}
i & 1 & 2 & 3 & 4 & \cdots &4h-3 & 4h-2 & 4h-1 & 4h\\\hline\hline
z_iz_{i+1} & 2n+1 & q-2n-3 & 2n+2 & q-2n-4 & \cdots & 2n+2h-1 & 2n+4h+2k+1 & 2n+2h & 2n+4h+2k  \\ \hline
\end{array}\]

We label $S_{2k}$ by integers in $[2n+2h+1,2n+2h+k]\cup[2n+4h+k+1,2n+4h+2k-1]$. If $k$ is even, we have
\[\T{6}\begin{array}{c|*{7}{|c}||c||}
i & 1 & 2 & 3 & 4 & \cdots & k-2 & k-1 & k \\\hline\hline
w_iw_{i+1} & 2n+2h+2 & 2n+4h+2k-2  & 2n+2h+4 & 2n+4h+2k-4 & \cdots & 2n+4h+k+2 & 2n+2h+k & 2n+4h+k+1 \\ \hline
\end{array}\]
\[\T{6}\begin{array}{c||c|*{7}{|c}|}
i & k & k+1 & k+2 & \cdots & 2k-4 & 2k-3 & 2k-2 & 2k-1\\\hline\hline
w_iw_{i+1} & 2n+4h+k+1 & 2n+2h+k-1 & 2n+4h+k+3 &\cdots & 2n+4h+2k-3  & 2n+2h+3 & 2n+4h+2k-1 & 2n+2h+1  \\ \hline
\end{array}\]

If $k$ is odd, we have
\[\T{6}\begin{array}{c|*{8}{|c}|}
i & 1 & 2 & 3 & 4 & \cdots & k-2 & k-1 & k \\\hline\hline
w_iw_{i+1} & 2n+2h+2 & 2n+4h+2k-2  & 2n+2h+4 & 2n+4h+2k-4 & \cdots & 2n+2h+k-1 & 2n+4h+k+1 & 2n+2h+k \\ \hline
\end{array}\]
\[\T{6}\begin{array}{c|*{8}{|c}|}
i & k & k+1 & k+2 & \cdots & 2k-4 & 2k-3 & 2k-2 & 2k-1\\\hline\hline
w_iw_{i+1} & 2n+2h+k & 2n+4h+k+2 & 2n+2h+k-2 &\cdots & 2n+4h+2k-3  & 2n+2h+3 & 2n+4h+2k-1 & 2n+2h+1  \\ \hline
\end{array}\]

We now have $P_{2n+4h+2k+1} = P^{(1)}_{2n+2}Q_{4h+1}S_{2k}$. Moreover, $P_{2n+2} = P^{(2)}_{2n+2}$. We shall now label the remaining $2h$ edge of $R_{2h+2}$ using integers in $[2n+2h+k+1,2n+4h+k]$.
Similar to the proof of the case~A, we require $x=y+2h-1$. That is, $r=2h-1$ and $N=2h$. By Corollary~\ref{cor-circ-perm}, we have a local antimagic $4$-labeling for $Sp(2n+1,2h+1,2n+4h+2k)$.
\end{enumerate}
\end{enumerate}
\end{proof}

%
%
%
%
%
%
%
%

\begin{theorem}\label{thm-eolsmall}
For $n\ge 1$ and $2\le l\le m+1$, $\chi_{la}(Sp(2n,2n+2m+1,l))=4$.
\end{theorem}

\begin{proof}\hspace*{\fill}{}
\begin{enumerate}[(a)]
\item Suppose $l=2h+1$ with $h<n$. Then $Sp(2n,2n+2m+1,2h+1)=Sp(2h+1, 2n, 2n+2m+1)=Sp(2h+1, 2h+2(n-h), 2n+2m+1)$. Clearly, $2n+2m+1\ge \max\{n-h, 3\}$. By Theorem~\ref{thm-oel}, $\chi_{la}(Sp(2n+1,2n+2m,l))=4$.
\item Suppose $l=2h+1$ with $h\ge n$. Then $Sp(2n,2n+2m+1,2h+1)=Sp(2n,2h+1,2n+2m+1)=Sp(2n,2n+2(h-n)+1,2n+2m+1)$. Since $2h+1\le m+1$, $2n+2m+1\ge (h-n)+2$. By Theorem~\ref{thm-eol}, $\chi_{la}(Sp(2n+1,2n+2m,l))=4$.

\item Suppose $l=2h$, $h\ge 1$. Since $l\le m+1$, we let $m=2h-1+2k$, $k\ge 0$, to get $Sp(2n,2h,2n+4h+4k-1)$. Now, $q=4n+6h+4k-1$.

   Suppose $h\ge 3$. Similar to the proof of Theorem~\ref{thm-eol}, let $P^{(1)}_{2n+1}=x_{1}\cdots x_{2n}x_{2n+1}$, $P^{(2)}_{2n+1}= y_{1}\cdots y_{2n+1}$, $Q_{4h+4k}=v_1\cdots v_{2h+4k-1}v_{4h+4k}$ and $P_{2h+1}=z_1\cdots z_{2h+1}$. Let $f$ be a required labeling. Firstly, let $f(z_1z_2)=q$, $f(z_2z_3)=x$ and $f(z_{l}z_{l+1})=y$.

We label $P^{(j)}_{2n+1}$, $1\le j\le 2$, by
\[\begin{array}{c|*{7}{|c}|}
i & 1 & 2 & 3 & 4 & \cdots & 2n-1 & 2n\\\hline\hline
x_ix_{i+1} & q-1 & 1& q-3 & 3 & \cdots & q-2n+1 & 2n-1\\\hline
y_iy_{i+1} & q-2 & 2 & q-4 & 4 & \cdots & q-2n & 2n\\ \hline
\end{array}\]

We label $Q_{4h+4k}$ by
\[\T{7}\begin{array}{c|*{9}{|c}|}
i & 1 & 2 & 3 & 4 & \cdots & 2h+2k-3 & 2h+2k-2 & 2h+2k-1 & 2h+2k\\\hline\hline
v_iv_{i+1} & q-2n-1 & 2n+1 & q-2n-3 & 2n+3 & \cdots & 2n+4h+2k+2 & 2n+2h+2k-3 & 2n+4h+2k & 2n+2h+2k-1 \\ \hline
\end{array}\]

\[\T{7}\begin{array}{c|*{8}{|c}|}
i & 2h+2k+1 & 2h+2k+2 & 2h+2k+3 & 2h+2k+4 & \cdots & 4h+4k-3 & 4h+4k-2 & 4h+4k-1 \\\hline\hline
v_{i}v_{i+1} & 2n+4h+2k-1 & 2n+2h+2k-2 & 2n+4h+2k+1 & 2n+2h+2k-4 & \cdots & q-2n-4 & 2n+2 & q-2n-2 \\ \hline
\end{array}\]

Let $P_{2n+1} = P^{(1)}_{2n+1}$ and $P_{2n+4h+4k}=P^{(2)}_{2n+1}Q_{4h+4k}$. We require $q+x = y + (2n-1) + (q-2n-2) = y+q-3$ so that $x+3=y$. Similar to the proof of Theorem~\ref{thm-2el}, we require $r=3$, $a=2n+2h+2k-1$ and $N=2h-1\ge 5$. By Corollary~\ref{cor-circ-perm}, we have a local antimagic $4$-labeling for $Sp(2n,2h,2n+4h+4k-1)$.

Suppose $h=1$ or $n=1$. By Theorem~\ref{thm-2el}, $\chi_{la}(Sp(2n,2h,2n+4h+4k-1))=4$.

Suppose $h=2$ and $n\ge 3$. We get $Sp(4,2n,2n+4k+7)$. Let $P^{(1)}_{5}=x_{1}\cdots x_{5}$, $P^{(2)}_{5}= y_{1}\cdots y_{5}$, $Q_{2n+4k+4}=v_1\cdots v_{2n+4k+4}$ and $P_{2n+1}=z_1\cdots z_{2n+1}$. Label the edges according to the steps above, we can get a similar conclusion.

If $n = h = 2$, we have $Sp(4,4,4k+11)$, $k\ge 0$. Let $P^{(1)}_5=x_1x_2\cdots x_5$, $P^{(2)}_5=y_1y_2\cdots y_5$ and $P^{(3)}_5=z_1z_2\cdots z_{5}$. Also let $Q_{4k+8}=v_1v_2\cdots v_{4k+8}$. We label $P^{(j)}_5$, $j=1,2,3$ by
\[\begin{array}{c|*{4}{|c}|}
i & 1 & 2 & 3 & 4 \\\hline\hline
x_ix_{i+1} & 4k+18 & 1 & 4k+16 & 3 \\ \hline
y_iy_{i+1} & 4k+17 & 2 & 4k+15 & 4  \\ \hline
z_iz_{i+1} & 4k+19 & 2k+8 & 2k+9 & 2k+10  \\ \hline
\end{array}\]

Label $Q_{4k+8}$ using integers in $[5,2k+7]\cup [2k+11,4k+14]$ by
\[\T{10}\begin{array}{c|*{9}{|c}|}
i & 1 & 2 & 3 & 4 & \cdots & 2k+1 & 2k+2 & 2k+3 & 2k+4 \\\hline\hline
v_iv_{i+1} & 4k+14 & 5 & 4k+12 & 7 & \cdots & 2k+14 & 2k+5 & 2k+12 & 2k+7 \\ \hline
\end{array}\]

\[\T{10}\begin{array}{c|*{8}{|c}|}
i & 2k+5 & 2k+6 & 2k+7 & 2k+8 & \cdots & 4k+5 & 4k+6 & 4k+7 \\\hline\hline
v_{i}v_{i+1} & 2k+11 & 2k+6 & 2k+13 & 2k+4 & \cdots & 4k+11 & 6 & 4k+13 \\ \hline
\end{array}\]

Let $P_{4k+12}=P^{(1)}_5Q_{4k+8}$ and the two paths of length 5 be $P^{(j)}_5, j=2,3$. It is easy to check that $Sp(4,4,4k+11)$ admits a local antimagic 4-labeling with vertex labels in $\{4k+17,4k+18,4k+19,6k+27\}$.
\end{enumerate}
\end{proof}

Combining Theorems~\ref{thm-eol}, \ref{thm-oel} and \ref{thm-eolsmall}, we have
\begin{theorem}\label{thm-2parity}
For $a,b,c\ge 2$ with even $a$ and odd $b$, $\chi_{la}(Sp(a,b,c))=4$.
\end{theorem}

\section{Spider with three even legs}

Now we are going to consider the case for three even legs.

\begin{theorem}\label{thm-3even} For $h\ge m\ge n\ge 1$, $\chi_{la}(Sp(2n,2m,2h))=4$. \end{theorem}

\begin{proof} By Corollary~\ref{cor-evenlegs} and Theorem~\ref{thm-2el}, we may assume $h> m> n>1$. Let $f: E(Sp(2n,2m,2h))\to [1,2n+2m+2h]$ be a bijection. Note that $q=2n+2m+2h$. Let $r=q/2=n+m+h$. Let $P_{2n+1} = x_1x_2\cdots x_{2n+1}$, $P_{2m+1} = y_1y_2\cdots y_{2m+1}$, $P_{2h+1} = z_1z_2\cdots z_{2h+1}$. We first consider $(m,h)\ne (n+1,n+2)$. Let $f(x_1x_2)=q-1$, $f(y_1y_2)=q-2$ and $f(z_1z_2)=q$.

Label the remaining edges of $P_{2n+1}$ as:
\[\begin{array}{c|*{9}{|c}|}
i & 2 & 3 & 4 & 5 &  \cdots & 2n-3 & 2n-2 & 2n-1 & 2n \\\hline\hline
x_ix_{i+1}  & r & r-2 & r+2 & r-4 &  \cdots  & r-2n+4 & r+2n-4 & r-2n+2 & r+2n-2 \\\hline
\end{array}\]

For $P_{2h+1}$, we label the subpath $z_2z_3\cdots z_{2n+2}$ as:
\[\begin{array}{c|*{10}{|c}|}
k & 2 & 3 & 4 & 5 & \cdots & 2n-2 & 2n-1 & 2n & 2n+1  \\\hline\hline
z_kz_{k+1} & r-1 & r+1 & r-3 & r+3 & \cdots & r-2n+3 & r+2n-3 & r-2n+1 & r+2n-1 \\\hline
\end{array}\]

For odd $m$, label the remaining edges of $P_{2m+1}$ as:
\[\T{9}\begin{array}{c|*{10}{|c}|}
j & 2 & 3 & 4 & 5  & \cdots & m-3& m-2 & m-1 & m & m+1 \\\hline\hline
y_iy_{j+1} & 1 & q-3 & 3 & q-5  & \cdots &m-4 & q-m+2 & m-2 & q-m & m \\\hline
y_{m+j}y_{m+1+j}  & q-m-1 & m-1 & q-m+1 & m-3 &\cdots & q-6 & 4 & q-4 & 2 &\\\hline
\end{array}\]
For even $m$, label the remaining edges of $P_{2m+1}$ as:
\[\T{9}\begin{array}{c|*{10}{|c}|}
j  & 2 & 3 & 4 & 5 & \cdots & m-3 & m-2 & m-1 & m & m+1 \\\hline\hline
y_iy_{j+1}  & 1 & q-3 & 3 & q-5 & \cdots & q-m+3 &  m-3 & q-m+1 & m-1 & q-m-1\\\hline
y_{m+j}y_{m+1+j} & m & q-m & m-2  & q-m+2 &\cdots & q-6 & 4 & q-4 & 2 &\\\hline
\end{array}\]

Up to now, $[1,m]\cup [r-2n+1, r+2n-1]\cup [q-m-1, q]$ are used. Note that, $r-2n+1\ge m+2$ and $r+2n-1\le 2n+m+2h-3=q-m-3$. So no label is used twice. Now we shall assign the labels in $[m+1, r-2n]\cup [r+2n, q-m-2]$ to the unlabeled edges of $P_{2h+1}$.

For $n\equiv h\pmod{2}$, we label the remaining edges of $P_{2h+1}$ as:
\[\begin{array}{c|*{11}{|c}|}
k & 2n+2 & 2n+3 & 2n+4 & 2n+5 & \cdots & n+h-2 & n+h-1  & n+h & n+h+1  \\\hline\hline
z_kz_{k+1} & r-2n & r+2n & r-2n-2 & r+2n+2 &\cdots & m+4 & q-m-4 & m+2 & q-m-2\\\hline
\end{array}\]
\[\T{9}\begin{array}{c|*{9}{|c}|}
k & 3 & 4 & 5 & 6 & \cdots & h-n-2 & h-n-1  & h-n & h-n+1  \\\hline\hline
z_{n+h-1+k}z_{n+h+k} & m+1 & q-m-3 & m+3 & q-m-5 & \cdots & r+2n+3 &r-2n-3 & r+2n+1 & r-2n-1 \\\hline
\end{array}\]
For $n\not\equiv h\pmod{2}$, we label the remaining edges of $P_{2h+1}$ as:
\[\begin{array}{c|*{11}{|c}|}
k & 2n+2 & 2n+3 &  2n+4 & 2n+5 & \cdots & n+h-2 & n+h-1  & n+h & n+h+1  \\\hline\hline
z_kz_{k+1} & r-2n & r+2n & r-2n-2 & r+2n+4 &\cdots & q-m-5 & m+3 & q-m-3 & m+1\\\hline
\end{array}\]
\[\T{9}\begin{array}{c|*{9}{|c}|}
k & 3 & 4 & 5 & 6 & \cdots & h-n-2 & h-n-1  & h-n & h-n+1  \\\hline\hline
z_{n+h-1+k}z_{n+h+k} & q-m-2 & m+2 & q-m-4 & m+4 & \cdots & r+2n+3 &r-2n-3 & r+2n+1 & r-2n-1 \\\hline
\end{array}\]
Let $u = x_{2n+1} = y_{2m+1} = z_{2h+1}$. When $(2n,2m,2h)\not=(2n,2n+2,2n+4)$. It is routine to check that $f^+(u) = 2r-1=q-1$, $f^+(x_{2n})=f^+(z_{2h})=q$, $f^+(y_{2m})=q-2$, $f^+(x_2)=f^+(z_2)=q+r-1$, and all other induced vertex labels are $q-2$, $q-1$ and $q$.
Thus, $f$ is a local antimagic 4-labeling.

\ms
Consider $Sp(2n, 2n+2, 2n+4)$, $n\ge 2$. Now $q=6n+6$. Let
$f(x_1x_2)=6n+5$, $f(y_1y_2)=6n+4$ and $f(z_1z_2)=6n+6$.
Label the remaining edges of $P_{2n+1}$ as:
\[\begin{array}{c|*{9}{|c}|}
i & 2 & 3 & 4 & 5 &  \cdots & 2n-3 & 2n-2 & 2n-1 & 2n \\\hline\hline
x_ix_{i+1}  & 3n+3 & 3n+1 & 3n+5 & 3n-1 &  \cdots  & n+7 & 5n-1 & n+5 & 5n+1 \\\hline
\end{array}\]
Label the remaining edges of $P_{2n+5}$ as:
\[\begin{array}{c|*{9}{|c}||c|c|}
k & 2 & 3 & 4 & 5 & \cdots & 2n-1 & 2n & 2n+1 & 2n+2 & 2n+3 & 2n+4  \\\hline\hline
z_kz_{k+1} & 3n+2 & 3n+4 & 3n & 3n+6 & \cdots & 5n & n+4 & 5n+2 & n+2 & 5n+3 & n+3 \\\hline
\end{array}\]
Now we used labels $[n+2, 5n+3]\cup\{6n+4, 6n+5, 6n+6\}$.

For even $n$, label the remaining edges of $P_{2n+3}$ as:
\[\begin{array}{c|*{9}{|c}|}
j  & 2 & 3 & 4 & 5 & \cdots  & n-1 & n & n+1 & n+2\\\hline\hline
y_iy_{j+1}  & 2 & 6n+2 & 4 & 6n & \cdots & 5n+6 &  n & 5n+4 &\\\hline
y_{n+j}y_{n+1+j} & n+1 & 5n+5 & n-1 & 5n+7 & \cdots & 6n+1 & 3 & 6n+3 & 1 \\\hline
\end{array}\]

For odd $n$, label the remaining edges of $P_{2n+3}$ as:
\[\begin{array}{c|*{9}{|c}|}
j  & 2 & 3 & 4 & 5 & \cdots  & n-1 & n & n+1 & n+2\\\hline\hline
y_iy_{j+1}  & 2 & 6n+2 & 4 & 6n & \cdots & n-1 &  5n+5 & n+1 &\\\hline
y_{n+j}y_{n+1+j} & 5n+4 & n & 5n+6 & n-2 & \cdots & 6n+1 & 3 & 6n+3 & 1 \\\hline
\end{array}\]

It is routine to check that $f^+(u) = 6n+5$, $f^+(x_{2n})=f^+(z_{2n+4})=6n+6$, $f^+(y_{2n+2})=6n+4$, $f^+(x_2)=f^+(z_2)=9n+8$, and all other induced vertex labels are $6n+4$, $6n+5$ and $6n+6$. Thus, $f$ is a local antimagic 4-labeling. The theorem holds.
\end{proof}

\begin{example} Following are labelings for some spiders of 3 even legs according to the proof above. The numbers listed in the parenthesis are the induced vertex colors.\\[-2mm]

\nt\begin{minipage}[b]{8.5cm}
$Sp(4,6,8)$:\\
$\begin{array}{*{8}{c}}
17 & 9 & 7 & 11\\
16 & 2 & 14 & \underline{3} & \underline{15} & \underline{1}\\
18 & 8 & 10 & 6 & 12 & 4 & \underline{13} & \underline{5}
\end{array}$\\
(26, 18, 17, 16)
\end{minipage}
\begin{minipage}[b]{8.5cm}
$Sp(4,6,12)$:\\
$\begin{array}{*{12}{c}}
21 & 11 & 9 & 13\\
20 & 1 & 19 & 3 & \underline{18} & \underline{2}\\
22 & 10 & 12 & 8 & 14 & \underline{7} & \underline{15} & \underline{5} & \underline{17} & 4 & 16 & 6
\end{array}$\\
(32, 22, 20, 19)
\end{minipage}\\

\nt\begin{minipage}[b]{8.5cm}
$Sp(4,8,10)$:\\
$\begin{array}{*{10}{c}}
21 & 11 & 9 & 13\\
20 & 1 & 19 & 3 & 17 & \underline{4} & \underline{18} & \underline{2}\\
22 & 10 & 12 & 8 & 14 & \underline{7} & \underline{15} & \underline{5} & 16 & 6
\end{array}$\\
(32, 22, 21, 20)
\end{minipage}
\begin{minipage}[b]{8.5cm}
$Sp(6,8,10)$:\\
$\begin{array}{*{10}{c}}
23 & 12 & 10 & 14 & 8 & 16\\
22 & 2 & 20 & 4 & \underline{19} & \underline{3} & \underline{21} & \underline{1}\\
24 & 11 & 13 & 9 & 15 & 7 & 17 & 5 & \underline{18} & \underline{6}
\end{array}$\\
(35, 24, 23, 22)
\end{minipage}
\rsq
\end{example}

\section{Spiders with three odd legs}

\begin{theorem}\label{thm-sp(2n+1,2m+1,n+m+3k+1)} For $m,n\ge 1, k\ge 1$, $m+n+k$ even and $3k\le n+m$, $\chi_{la}(Sp(2n+1,2m+1,n+m+3k+1))=4$.  \end{theorem}

\begin{proof} Since $m+n\ge 3k$, without loss of generality, we assume that $m>k$. Let $f$ be an edge labeling of $Sp(2n+1,2m+1,n+m+3k+1)$. Let $P_{2n+2}=x_{1}\cdots x_{2n+1}x_{2n+2}$, $P_{2m+2}= y_{1}\cdots y_{2m+1}y_{2m+2}$ and $P_{n+m+3k+2}=z_1\cdots z_{n+m+3k+1}z_{n+m+3k+2}$. Let $u=x_{2n+2}=y_{2m+2}=z_{n+m+3k+2}$. Now, $q=3n+3m+3k+3$. We label $P_{2n+2}$, $P_{2m+2}$ and $P_{n+m+3k+2}$ as:

{\T{8}
\noindent$\begin{array}{c||c|c|c|}
i & 1 & 2 & 3 \\\hline\hline
x_ix_{i+1} & 3n+3m+3k+3 & n+m+k & 2n+2m+2k+1 \\\hline
\end{array}$\\[1mm]
\noindent$\begin{array}{c||*{7}{c|}}
i & 4  & 5 & 6 & 7 &\cdots & 2n & 2n+1  \\\hline\hline
x_ix_{i+1} & 2n+2m+2k+2 & n+m+k-1 & 2n+2m+2k+3 & n+m+k-2& \cdots &  3n+2m+2k & m+k+1   \\\hline
\end{array}$}\\[1mm]
Used labels: $[m+k+1, n+m+k]\cup[2n+2m+2k+1, 3n+2m+2k]\cup\{3m+3n+3k+3\}$.

{\T{8}
\noindent$\begin{array}{c|*{8}{|c}|}
i & 1 & 2 & 3 & 4 & \cdots & 2k-1&  2k & 2k+1 \\\hline\hline
y_iy_{i+1} & 3n+3m+3k+1 & 2 & 3n+3m+3k-1 & 4  &\cdots & 3n+3m+k+3 & 2k & 3n+3m+k+1 \\\hline
\end{array}$\\[1mm]
\noindent$\begin{array}{c|*{7}{c|}}
i & 2k+2 &2k+3 & 2k+4 & 2k+5 & \cdots & 2m & 2m+1 \\\hline\hline
y_iy_{i+1}  & 2k+1 & 3n+3m+k & 2k+2 & 3n+3m+k-1 & \cdots & m+k	& 3n+2m+2k+1 \\\hline
\end{array}$}\\[1mm]
Used labels: $([2, 2k]\cap \mathbb E)\cup[2k+1, m+k]\cup [3n+2m+2k+1, 3n+3m+k]\cup([3n+3m+k+1, 3n+3m+3k+1]\cap \mathbb{O})$, where $\mathbb E$ and $\mathbb O$ are the sets of even and odd integers, respectively.

{\T{8}
\noindent$\begin{array}{c|*{8}{|c}|}
i & 1 & 2 & 3 & 4 & \cdots & 2k-1&  2k & 2k+1 \\\hline\hline
z_iz_{i+1} & 3n+3m+3k+2 & 1 & 3n+3m+3k & 3  &\cdots & 3n+3m+k+4 & 2k-1 & 3n+3m+k+2 \\\hline
\end{array}$\\[1mm]
\noindent$\begin{array}{c||*{7}{c|}}
i & 2k+2 &2k+3 & 2k+4 & 2k+5 & \cdots & n+m+2-k & n+m+3-k \\\hline\hline
z_iz_{i+1}  & n+m+3k+1 & 2n+2m & n+m+3k+3 & 2n+2m-2 & \cdots & 2n+2m+1	& n+m+3k \\\hline
\end{array}$\\[1mm]		
\noindent$\begin{array}{c||*{7}{c|}}
i & n+m+4-k & n+m+5-k &  n+m+4-k & n+m+5-k &\cdots & n+m+3k & n+m+3k+1 \\\hline\hline
z_iz_{i+1}  & 2n+2m+2& n+m+3k-1 &2n+2m+3&  n+m+3k-2 &\cdots & 2n+2m+2k& n+m+k+1	\\\hline
\end{array}$}\\[1mm]
Used labels: $([1, 2k-1]\cap \mathbb O)\cup [n+m+k+1, 2n+2m+2k]\cup([3n+3m+k+2, 3n+3m+3k+2]\cap \mathbb{E})$.			

Clearly $f$ is a bijection. Now $f^+(x_2) = f^+(x_4)= f^+(z_{2k+2})=f^+(u) = 4n+4m+4k+3$ and\\ $f^+(w)\in\{3n+3m+3k+3, 3n+3m+3k+2, 3n+3m+3k+1\}$ for other vertex $w$. Thus, $f$ is a required local antimagic 4-coloring.
\end{proof}

\begin{example} $Sp(17,15,25)$ ($n=8, m=7, k=3$):\\
{\T{10}
$\begin{array}{*{25}{c}}
57 & 18	& 37; & 38 & 17 & 39 & 16 & 40 & 15 & 41 & 14 & 42 & 13 & 43	& 12 & 44 &11 \\
55 & 2 & 53 & 4 & 51 & 6 & 49; & 7 & 48 & 8 & 47 & 9 & 46 & 10 & 45\\
56& 1 & 54 & 3& 52 & 5 & 50; & 25 & 30 & 27 & 28 & 29 & 26 & 31 & 24; & 32 & 23 & 33 & 22	& 34 & 21 & 35 & 20	& 36& 19
\end{array}$}\\
$(75,57,56,55)$
\end{example}

\begin{theorem}\label{thm-sp(2n+1,2m+1,n+m+1)} For $m\ge 2,n\ge 1$ and $m+n\ge 4$ even, $\chi_{la}(Sp(2n+1,2m+1,n+m+1))=4$.  \end{theorem}

\begin{proof}Let $f$ be a bijective edge labeling of $Sp(2n+1,2m+1,n+m+1)$. Let $P_{2n+2}=x_{1}\cdots x_{2n+1}x_{2n+2}$, $P_{2m+2}= y_{1}\cdots y_{2m+1}y_{2m+2}$ and $P_{n+m+2}=z_1\cdots z_{n+m+1}z_{n+m+2}$. Let $u=x_{2n+2}=y_{2m+2}=z_{n+m+2}$. Now, $q=3n+3m+3$. We label $P_{2n+2}$, $P_{2m+2}$ and $P_{n+m+2}$ as:
{\T{8}
\begin{align*}&\begin{array}{c|*{8}{|c}|}
i & 1 & 2 & 3 & 4 & 5 & \cdots & 2n & 2n+1  \\\hline\hline
x_ix_{i+1} & 3n+3m+2 & 1 & 3n+3m & 2 & 3n+3m-1  & \cdots &  n & 2n+3m+1   \\\hline
\end{array}\\
&\begin{array}{c|*{3}{|c}"*{7}{c|}}
i & 1 & 2 & 3 & 4 & 5 & 6 & 7 & \cdots  & 2m & 2m+1 \\\hline\hline
y_iy_{i+1} & 3n+3m+3 & n+m & 2n+2m+1 & 2n+2m+2 & n+m-1 & 2n+2m+3 & n+m-2 & \cdots  & 2n+3m & n+1   \\\hline
\end{array}\\
&\begin{array}{c|*{8}{|c}|}
i & 1 & 2 & 3 & 4 & 5 & \cdots  & n+m & n+m+1   \\\hline\hline
z_iz_{i+1} & 3n+3m+1 & n+m+2 & 2n+2m-1 & n+m+4 & 2n+2m-3 & \cdots &  2n+2m & n+m+1  \\\hline
\end{array}\end{align*}}

Now $f^+(y_4) = f^+(u) = 4n+4m+3$ and $f^+(w)\in\{3n+3m+3, 3n+3m+2, 3n+3m+1\}$ for other vertex $w$. Thus, $f$ is a required local antimagic 4-coloring.
\end{proof}

\newpage
\begin{example} $Sp(9,17,13)$:\\
$\begin{array}{*{17}{c}}
38 & 1 & 36 & 2 & 35 & 3 & 34 & 4 & 33 \\
39 & 12 & 25 & 26 & 11 & 27 & 10 & 28 & 9 & 29 & 8 & 30 & 7 & 31 & 6 & 32 & 5 \\
37 & 14 &  23 & 16 & 21 & 18 & 19 & 20 &  17 & 22 & 15 &  24 &  13 \\
\end{array}$\\
$(37,38,39,51)$ \rsq
\end{example}

By substituting $m=n+2$ or $m=kn$, where $k\ge 1$ and $(k+1)n$ is even, in Theorem~\ref{thm-sp(2n+1,2m+1,n+m+1)}, we have
\begin{corollary}\label{cor-(2n+1,2n+3,2n+5)} For $n,k\ge 1$, $\chi_{la}(Sp(2n+1,2n+3,2n+5))=\chi_{la}(Sp(2n+1, 2kn+1,(k+1)n+1))=4$, where $(k+1)n$ is even.
\end{corollary}

\begin{theorem}\label{thm-equalodd} For $n\ge 0$ and $m\ge 1$, $\chi_{la}(Sp(2n+1,2m+1,2m+1))=4$.
\end{theorem}

\begin{proof} Let $f$ be a required labeling.
Let $P_{2n+2}=x_{1}\cdots x_{2n+1}x_{2n+2}$, $P_{2m+2}= y_{1}\cdots y_{2m+1}y_{2m+2}$ and $P_{2h+2}=z_1\cdots z_{2h+1}z_{2h+2}$. Let $u=x_{2n+2}=y_{2m+2}=z_{2h+2}$.
 We label $P_{2n+2}$ and $P_{2m+2}$ as:
\[\begin{array}{c|*{8}{|c}|}
i & 1 & 2 & 3 & 4 & \cdots & 2n-1 & 2n & 2n+1\\\hline\hline
x_ix_{i+1} & q-1 & 1 & q-2 & 2 & \cdots & q-n & n & q-n-1\\\hline
\end{array}\]
When $n=0$. The above table has only one column.
\[\begin{array}{c|*{8}{|c}|}
i & 1 & 2 & 3 & 4 & \cdots & 2m-1 & 2m & 2m+1\\\hline\hline
y_iy_{i+1} & n+1 & q-n-2 & n+2 & q-n-3 & \cdots & n+m & q-n-m-1 & n+m+1\\\hline
\end{array}\]
\[\begin{array}{c||c|*{6}{|c}|}
i & 1 &  2 & 3 & \cdots & 2h-1 & 2h & 2h+1\\\hline\hline
z_iz_{i+1} & q& q-n-m-2 & n+m+2 &  \cdots & n+m+h & q-n-m-h-1 & n+m+h+1\\\hline
\end{array}\]
Now $f^+(u)=q+n+2m+h+1=3n+4m+3h+4$, $f^+(z_2)=2q-n-m-2=3n+3m+4h+4$, $f^+(y_1)=n+1$ and $f^+(w)\in\{q, q-1\}$ for other vertex $w$. Thus, when $m=h$ we have the result.\end{proof}

\begin{example}\hspace*{\fill}{}\\[-4mm]

\nt\begin{minipage}[b]{8.5cm}
$Sp(5,5,9)\cong Sp(9,5,5)$:\\
$\begin{array}{*{9}{c}}
18 & 1 & 17 & 2 & 16 & 3 & 15 & 4 & 14\\
5 & 13 & 6 & 12 & 7\\
19 & 11 & 8 & 10 & 9
\end{array}$\\
(30, 19, 18, 5)
\end{minipage}
\begin{minipage}[b]{8.5cm}
$Sp(7,9,9)$:\\
$\begin{array}{*{11}{c}}
24 & 1 & 23 & 2 & 22 & 3 & 21\\
4 & 20 & 5 & 19 & 6 & 18 & 7& 17& 8\\
25 & 16& 9 & 15& 10& 14& 11& 13& 12\\
\end{array}$\\
(41, 25, 24, 4)
\end{minipage}\rsq
\end{example}

Starting from the labeling of $Sp(2n+1,2m+1,2m+1)$ defined in the proof of Theorem~\ref{thm-equalodd} for $m >n$, we move the $2n+2$ right most numbers of the second leg to the end of first leg. We then can get a required labeling for $Sp(4n+3, 2m-2n-1,2m+1)$.

Namely, change the following assignment
\[\begin{array}{c|*{7}{|c}|}
i & 1 & 2 & 3 & \cdots & 2n-1 & 2n & 2n+1\\\hline\hline
x_ix_{i+1} & q-1 & 1 & q-2 & \cdots & q-n & n & q-n-1\\\hline
\end{array}\]
\[\begin{array}{c|*{4}{|c}|*{5}{|c}|}
i & 1 & 2 & \cdots & 2(m-n)-1 & 2(m-n) & \cdots & 2m-1 & 2m & 2m+1\\\hline\hline
y_iy_{i+1} & n+1 & q-n-2 & \cdots & m & q-m-1 &  \cdots & n+m & q-n-m-1 & n+m+1\\\hline
\end{array}\]
to
\[\T{10}\begin{array}{c|*{7}{|c}|*{5}{|c}|}
j & 1 & 2 & 3 & \cdots & 2n-1 & 2n & 2n+1 & 2n+2 & \cdots & 4n+1 & 4n+2 & 4n+3\\ \hline\hline
v_jv_{j+1} & q-1 & 1 & q-2 & \cdots & q-n & n & q-n-1 & q-m-1 &  \cdots & n+m & q-n-m-1 & n+m+1\\\hline
\end{array}\]
\[\begin{array}{c|*{4}{|c}|}
i & 1 & 2 & \cdots & 2(m-n)-1 \\\hline\hline
w_iw_{i+1} & n+1 & q-n-2 & \cdots & m\\\hline
\end{array}\]
and keep the labeling of the leg $P_{2m+2}=z_1\cdots z_{2m+2}$ as
\[\begin{array}{c||c|*{6}{|c}|}
j & 1 &  2 & 3 & \cdots & 2m-1 & 2m & 2m+1\\\hline\hline
z_jz_{j+1} & q& q-n-m-2 & n+m+2 &  \cdots & n+2m & q-n-2m-1 & n+2m+1\\\hline
\end{array}\] Hence we have a labeling, say $g$, for $P_{4n+4}=v_1\cdots v_{4n+4}$, $P_{2m-2n}=w_1\cdots w_{2m-2n}$ and $P_{2m+2}=z_1\cdots z_{2m+2}$.
Now $g^+(u)=2n+4m+2=q-1$, $g^+(z_2)=2q-n-m-2=3n+7m+4=g^+(v_{2n+2})$, $g^+(y_1)=n+1$ and $g^+(v)\in\{q, q-1\}$ for other vertex $v$. So we have a local antimagic 4-labeling for $Sp(2(m-n)-1, 2m+1, 4n+3)$. After rewriting the parameters we have
\begin{theorem}\label{thm-oddspecial} For $m>  n\ge 0$, $\chi_{la}(Sp(2n+1, 2m+1, 4(m-n)-1))=4$.
\end{theorem}

\begin{example}\hspace*{\fill}{}\\[-4mm]

\nt\begin{minipage}[b]{7cm}\T{9}
Starting from the labeling of $Sp(3,9,9)$:\\
$\begin{array}{*{9}{c}}
20 &  1 & 19\\
2 & 18&  3& 17& 4& \underline{16}& \underline{5}& \underline{15}& \underline{6}\\
21& 14& 7& 13& 8& 12& 9& 11& 10
\end{array}$\\\\
we get $Sp(7,5,9) \cong Sp(5,7,9)$: \\
\nt
$\begin{array}{*{9}{c}}
20& 1& 19& \underline{16}& \underline{5}& \underline{15}& \underline{6}\\
2 & 18& 3& 17& 4\\
21& 14& 7& 13& 8& 12& 9& 11& 10
\end{array}$\\
(2, 10, 21, 35)
\end{minipage}
\begin{minipage}[b]{11cm}\T{9}
Starting from the labeling of $Sp(7,13,13)$:\\
$\begin{array}{*{13}{c}}
32& 1& 31& 2& 30& 3& 29\\
4& 28& 5& 27& 6& \underline{26}& \underline{7}& \underline{25}& \underline{8}& \underline{24}& \underline{9}& \underline{23}& \underline{10}\\
33& 22& 11& 21& 12& 20& 13& 19& 14& 18& 15& 17& 16
\end{array}$\\\\
we get $Sp(15,5,13) \cong Sp(5,13,15)$: \\
$\begin{array}{*{15}{c}}
32& 1& 31& 2& 30& 3& 29& \underline{26}& \underline{7}& \underline{25}& \underline{8}& \underline{24}& \underline{9}& \underline{23}& \underline{10}\\
4& 28& 5& 27& 6\\
33& 22& 11& 21& 12& 20& 13& 19& 14& 18& 15& 17& 16
\end{array}$\\
(4, 32, 33, 55)
\end{minipage}\rsq
\end{example}



\begin{theorem}\label{thm-Sp(2m+1,2m+3,4m-3)} For $m\ge 4$, $\chi_{la}(Sp(2m+1,2m+3,4m-3)) = 4$. \end{theorem}

\begin{proof} Let $f$ be a bijective edge labeling of $Sp(2m+1,2m+3,4m-3)$. Let $P_{2m+2}=x_{1}\cdots x_{2m+1}x_{2m+2}$, $P_{2m+4}= y_{1}\cdots y_{2m+3}y_{2m+4}$ and $P_{4m-2}=z_1\cdots z_{4m-3}z_{4m-2}$. Let $u=x_{2m+2}=y_{2m+4}=z_{4m-2}$. Now, $q=8m+1$. We label $P_{2m+2}$, $P_{2m+4}$ and $P_{4m-2}$ as:

\begin{align*}&\begin{array}{c|*{8}{|c}|}
i & 1 & 2 & 3 & 4 & 5 & \cdots & 2m & 2m+1  \\\hline\hline
x_ix_{i+1} & 8m+1 & 6m-1 & 2m+2 & 6m-2 & 2m+3  & \cdots &  5m & 3m+1   \\\hline
\end{array}\\
&\begin{array}{c|*{7}{|c}"*{4}{c|}}
i & 1 & 2 & 3 & 4 & \cdots & 2m-2 & 2m-1 &2m & 2m+1 & 2m+2 & 2m+3 \\\hline\hline
y_iy_{i+1} & 8m & 1 & 8m-1 & 2 & \cdots & m-1 & 7m+1&7m-1 & m+1 & 7m & m   \\\hline
\end{array}\\
&\begin{array}{c|*{7}{|c}|}
i & 1 & 2 & 3 & 4 & \cdots & 2m-5 & 2m-4   \\\hline\hline
z_iz_{i+1} & 6m+1 & 2m-1 & 6m+2 & 2m-2 & \cdots & 7m-2 & m+2  \\\hline
\end{array}\\
&\begin{array}{c|*{8}{|c}|}
i & 2m-3 & 2m-2 & 2m-1 & 2m & \cdots & 4m-9 & 4m-8 & 4m-7 \\\hline\hline
z_iz_{i+1} & 5m-1 & 3m+2 & 5m-2 & 3m+3 & \cdots & 4m+2 & 4m-1 & 4m+1 \\\hline
\end{array}\\
&\begin{array}{c|*{4}{|c}|}
i & 4m-6 & 4m-5 & 4m-4 & 4m-3 \\\hline\hline
z_iz_{i+1} & 4m & 2m+1 & 6m & 2m \\\hline
\end{array}\end{align*}

Now $f^+(x_2) = f^+(u) = 14m$ and $f^+(w)\in\{8m+1, 8m, 6m+1\}$ for other vertex $w$. Thus, $f$ is a required local antimagic 4-coloring.
\end{proof}

\begin{example} $Sp(11,13,17)$:\\
$\begin{array}{*{17}{c}}
41  & 29  & 12  & 28  & 13  & 27  & 14  & 26  & 15  & 25  & 16  \\
40  & 1  & 39  & 2  & 38  & 3  & 37  & 4  & 36  & 34 & 6 & 35 & 5\\
31  & 9  & 32  & 8  & 33  & 7  & 24  & 17  & 23  & 18  & 22  & 19  & 21 & 20 & 11 & 30 & 10 \\
\end{array}$\\
$(31,40,41,70)$ \rsq
\end{example}

Observe that if we move the last four labels of $P_{2m+2}$ (in the proof of Theorem~\ref{thm-Sp(2m+1,2m+3,4m-3)}) to the right end of $P_{2m+4}$ to form a labeling of $P_{2m+8}$, we then obtain a local antimagic 4-labeling for $Sp(2m-3,2m+7,4m-3)$. Renaming the parameters, we get the following theorem.

\begin{theorem} For $m\ge 2$, $\chi_{la}(Sp(2m+1,2m+11,4m+5)) = 4$. \end{theorem}

\begin{proof} Let $f$ be a bijective edge labeling of $Sp(2m+1,2m+11,4m+5)$. Now, $q = 8m+17$. Let $u = x_{2m+2} = y_{2m+12} = z_{4m+6}$. The labelings of $P_{2m+2}=x_1\cdots x_{2m+2}$, $P_{2m+12}=y_1\cdots y_{2m+12}$ and $P_{4m+6} = z_1\cdots z_{4m+6}$ are given as follows.

\begin{align*}&\begin{array}{c|*{8}{|c}|}
i & 1 & 2 & 3 & 4 & 5 & \cdots & 2m & 2m+1  \\\hline\hline
x_ix_{i+1} & 8m+17 & 6m+11 & 2m+6 & 6m+10 & 2m+7  & \cdots &  5m+12 & 3m+5   \\\hline
\end{array}\\
&\begin{array}{c|*{7}{|c}|}
i & 1 & 2 & 3 & 4 & \cdots & 2m+2 & 2m+3  \\\hline\hline
y_iy_{i+1} & 8m+16 & 1 & 8m+15 & 2 & \cdots & m+1 & 7m+15  \\\hline
\end{array}\\
&\begin{array}{c|*{8}{|c}|}
i & 2m+4 & 2m+5 & 2m+6 & 2m+7 & 2m+8 & 2m+9 & 2m+10 & 2m+11 \\\hline\hline
y_iy_{i+1} & 7m+13 & m+3 & 7m+14 & m+2 & 5m+11 & 3m+6 & 5m+10 & 3m+7 \\\hline
\end{array}\\
&\begin{array}{c|*{7}{|c}|}
i & 1 & 2 & 3 & 4 & \cdots & 2m-1 & 2m   \\\hline\hline
z_iz_{i+1} & 6m+13 & 2m+3 & 6m+14 & 2m+2 & \cdots & 7m+12 & m+4  \\\hline
\end{array}\\
&\begin{array}{c|*{8}{|c}|}
i & 2m+1 & 2m+2 & 2m+3 & 2m+4 & \cdots & 4m-1 & 4m & 4m+1 \\\hline\hline
z_iz_{i+1} & 5m+9 & 3m+8 & 5m+8 & 3m+9 & \cdots & 4m+10 & 4m+7 & 4m+9 \\\hline
\end{array}\\
&\begin{array}{c|*{4}{|c}|}
i & 4m+2 & 4m+3 & 4m+4 & 4m+5 \\\hline\hline
z_iz_{i+1} & 4m+8 & 2m+5 & 6m+12 & 2m+4 \\\hline
\end{array}\end{align*}

Now, $f^+(x_2) = f^+(y_{2m+4}) = 14m+28$  and $f^+(w)\in\{8m+17, 8m+16, 6m+1\}$ for other vertex $w$. Thus, $f$ is a required local antimagic 4-coloring.
\end{proof}

\begin{example} $Sp(9,19,21)$:\\
$\begin{array}{*{21}{c}}
49 & 35 & 14 & 34 & 15 & 33 & 16 & 32 & 17\\
48 & 1 & 47 & 2 & 46 & 3 & 45 & 4 & 44 & 5 & 43 & 41 & 7 & 42 & 6 & 31 & 18 & 30 & 19\\
37 & 11 & 38 & 10 & 39 & 9 & 40 & 8 & 29 & 20 & 28 & 21 & 27 & 22 & 26 & 23 & 25 & 24 & 13 & 36 & 12\\
\end{array}$\\
$(37,48,49,84)$ \rsq
\end{example}

\begin{theorem}\label{thm-3oo} For $h\ge 0$, $m\ge 1$, $\chi_{la}(Sp(3,2m+1,2m+2h+1))=4$.
\end{theorem}

\begin{proof} Let $f$ be a required labeling. We first consider $h\ge 2$.
Let $P^{(1)}_{2m+2}=x_{1}\cdots x_{2m+1}x_{2m+2}$, $P^{(2)}_{2m+2}= y_{1}\cdots y_{2m+1}y_{2m+2}$ and $P_4=w_1w_2w_3w_4$. Firstly, let $f(w_1w_2)=q$.

 We label $P^{(j)}_{2m+2}$ by integers in $[1,2m]\cup [q-2m-2,q-1]$:
\[\begin{array}{c|*{8}{|c}|}
i & 1 & 2 & 3 & 4 & \cdots & 2m-1 & 2m & 2m+1\\\hline\hline
x_ix_{i+1} & q-1 & 1& q-3 & 3 & \cdots & q-2m+1 & 2m-1 & q-2m-1\\\hline
y_iy_{i+1} & q-2 & 2 & q-4 & 4 & \cdots & q-2m & 2m & q-2m-2 \\ \hline
\end{array}\]

Suppose $h=2k$ and $k\ge 1$. Now, we have $Sp(3,2m+1,2m+4k+1)$ and $q=4m+4k+5$. Let $Q_{2k+1}=z_1\cdots z_{2k+1}$ and $R_{2k+1}=v_1\cdots v_{2k+1}$. We label $Q_{2k+1}$ by in integers in $[2m+1,2m+k]\cup [2m+3k+3,2m+4k+2=q-2m-3]$:
\[\T{9}\begin{array}{c|*{9}{|c}|}
i & 1 & 2 & 3 & 4 & \cdots &2k-3 & 2k-2 & 2k-1 & 2k\\\hline\hline
z_iz_{i+1} & 2m+1 & q-2m-3 & 2m+2 & q-2m-4 & \cdots & 2m+k-1 & 2m+3k+4 & 2m+k & 2m+3k+3  \\ \hline
\end{array}\]

We label $R_{2k+1}$ by integers in $[2m+k+2,2m+3k+1]$:
\[\T{9}\begin{array}{c|*{9}{|c}|}
i & 1 & 2 & 3 & 4 & \cdots &2k-3 & 2k-2 & 2k-1 & 2k\\\hline\hline
v_iv_{i+1} & 2m+k+2 & 2m+3k+1 & 2m+k+3 & 2m+3k & \cdots & 2m+k-1 & 2m+2k+3 & 2m+2k+1 & 2m+2k+2  \\ \hline
\end{array}\]

Let $P_{2m+2h+2}=P^{(2)}_{2m+2}Q_{2k+1}R_{2k+1}$. Lastly, let $f(w_2w_3)=2m+3k+2$ and $f(w_3w_4)=2m+k+1$. It is easy to check that $f^+(u)=q+2m+3k+2  =f^+(w_2)$ and each other vertex has induced label in $\{q-2,q-1,q\}$. Thus, $f$ is a local antimagic 4-labeling.

Suppose $h=2k+1$ and $k\ge 1$. Now, we have $Sp(3,2m+1,2m+4k+3)$ and $q=4m+4k+7$.  Let $Q_{2k+3}=z_1\cdots z_{2k+3}$ and $R_{2k+1}=v_1\cdots v_{2k+1}$.  We label $Q_{2k+3}$ by in integers in $[2m+1,2m+k+1]\cup [2m+3k+4,2m+4k+4=q-2m-3]$:
\[\T{9}\begin{array}{c|*{9}{|c}|}
i & 1 & 2 & 3 & 4 & \cdots &2k-1 & 2k & 2k+1 & 2k+2\\\hline\hline
z_iz_{i+1} & 2m+1 & q-2m-3 & 2m+2 & q-2m-4 & \cdots & 2m+k & 2m+3k+5 & 2m+k+1 & 2m+3k+4\\ \hline
\end{array}\]

We label $R_{2k+1}$ by integers in $[2m+k+3,2m+3k+2]$:
\[\T{8}\begin{array}{c|*{9}{|c}|}
i & 1 & 2 & 3 & 4 & \cdots &2k-3 & 2k-2 & 2k-1 & 2k\\\hline\hline
v_iv_{i+1} & 2m+k+3 & 2m+3k+2 & 2m+k+4 & 2m+3k+1 & \cdots & 2m+k+1 & 2m+2k+4 & 2m+2k+2 & 2m+2k+3  \\ \hline
\end{array}\]

Let $P_{2m+2h+2}=P^{(1)}_{2m+2}Q_{2k+3}R_{2k+1}$. Lastly, let $f(w_2w_3)=2m+3k+3$ and $f(w_3w_4)=2m+k+2$. It is easy to check that $f^+(u)=q+2m+3k+3 =f^+(w_2)$ and each other vertex has induced label in $\{q-2,q-1,q\}$. Thus, $f$ is a local antimagic 4-labeling.

For $h=1$, let $P_{2m+4}=x_{1}\cdots x_{2m+4}$, $P_{2m+2} =y_{1}\cdots y_{2m+2}$ and $P_4=w_1w_2w_3w_4$. Now $q=4m+7$.
We label $P_{2m+4}$ and $P_{2m+2}$ as:
\[\T{8}\begin{array}{c|*{13}{|c}|}
i & 1 & 2 & 3 & 4 & 5 & 6 & 7 & \cdots & 2m-1& 2m & 2m+1 & 2m+2 & 2m+3\\\hline\hline
x_ix_{i+1} & 4m+7 & 4m+6 & 1& 2m+1 & 2m+6 & 2m-1 & 2m+8 & \cdots & \cdots & 5 & 4m+2 & 3 & 4m+4\\\hline
y_iy_{i+1} & 4m+5 & 2 & 4m+3 & 4 & 4m+1 & \cdots & \cdots &\cdots &2m+7 & 2m & 2m+5 & &\\\hline
\end{array}\]
Let $f(w_1w_2)=2m+2$, $f(w_2w_3)=2m+3$ and $f(w_3w_4)=2m+4$.
Then the induced vertex labels are $8m+13$, $4m+7$, $4m+5$ and $2m+2$. Thus, $f$ is a local antimagic 4-labeling.

For $h=0$, let $P_{2m+2}=x_{1}\cdots x_{2m+2}$, $Q_{2m+2} =y_{1}\cdots y_{2m+2}$ and $P_4=w_1w_2w_3w_4$.

We label $P_{2m+2}$ and $Q_{2m+2}$ by
\[\begin{array}{c|*{10}{|c}|}
i & 1 & 2 & 3 & 4 & 5 & \cdots & 2m-2& 2m-1 & 2m & 2m+1\\\hline\hline
x_ix_{i+1} & 4m+2 & 2 & 4m & 4 & 4m-2 &\cdots &  2m-2 & 2m+4 & 2m & 2m+2\\\hline
y_iy_{i+1} & 4m+5 & 4m+1 & 3 & 4m-1 & 5 & \cdots &  2m+5 & 2m-1 & 2m+3 & 2m+1\\\hline
\end{array}\]
Let $f(w_1w_2)=4m+4$, $f(w_2w_3)=1$ and $f(w_3w_4)=4m+3$. Thus, $f$ is a local antimagic 4-labeling.

Hence the proof is completed.
\end{proof}

\begin{theorem}\label{thm-5oo} For $h\ge 3$, $\chi_{la}(Sp(5,2m+1,2h+1))=4$. \end{theorem}

\begin{proof} Let $P_6 = x_1\ldots x_6$, $P_{2m+2}=y_1\cdots y_{2m+2}$ and $P_{2h+2} = z_1z_2\cdots z_{2h+2}$ with $u=x_6=y_{12}=z_{2h+2}$. Now, $q=2m+2h+7$. Let $f$ be a required labeling.

The leg $P_6$ is labeled as follows.
\[\begin{array}{c|*{5}{|c}|}
i & 1 & 2 & 3 & 4 & 5 \\\hline\hline
x_ix_{i+1} & q-1 & 1 & q-2 & \frac{q+1}{2} & \frac{q-1}{2}  \\\hline
\end{array}\]
The leg $P_{2m+2}$ is labeled as follows.
\[\begin{array}{c|*{8}{|c}|}
j & 1 & 2 & 3 & 4 & 5 & \cdots & 2m & 2m+1 \\\hline\hline
y_jy_{j+1} & q & \frac{q-3}{2} & \frac{q+3}{2} & \frac{q-5}{2} & \frac{q+5}{2} & \cdots & \frac{q-1}{2}-m & \frac{q+1}{2}+m  \\\hline
\end{array}\]
The leg $P_{2h+2}$ is labeled as follows.
\[\begin{array}{c|*{7}{|c}|}
k & 1 & 2 & 3 & 4 & \cdots & 2h & 2h+1   \\\hline\hline
z_kz_{k+1} & 2 & q-3 & 3 & q-4 & \cdots & q-h-2 & h+2 \\\hline
\end{array}\]
Note that $h+2=\frac{q-1}{2}-m-1$ and $q-h-2=\frac{q+1}{2}+m+1$. So $f$ is a bijection. Now, $f^+(u) = f^+(x_4) = f^+(y_2) = \frac{3q-3}{2}$, $f^+(z_1) = 2$,  and $f^+(w)\in\{q, q-1\}$ for other vertex $w$. The theorem holds.
\end{proof}

\begin{example}

$Sp(5,9,11)$:\\
$\begin{array}{*{11}{c}}
24 & 1 & 23 & 13 & 12\\
25 & 11 & 14 & 10 & 15 & 9 & 16 & 8 & 17\\
2 & 22 & 3 & 21 & 4 & 20 & 5 & 19 & 6 & 18 & 7 \\
\end{array}$\\
(36, 25, 24, 2) \rsq
\end{example}

\begin{theorem}\label{thm-7oo} For $m,h\ge 1$, $\chi_{la}(Sp(7,2m+1,2h+1))=4$.   \end{theorem}
\begin{proof}
Let $P_8 = x_1\cdots x_8$, $P_{2m+2}=y_1\cdots y_{2m+2}$, $P_{2h+1} = z_1\cdots z_{2h+2}$, and $u=x_8=y_{2m+2}=z_{2h+2}$. Now, $q=2m+2h+9$. Let $f$ be a required labeling.\\
The leg $P_{8}$ is labeled as follows:
\[\begin{array}{c|*{7}{|c}|}
i & 1 & 2 & 3 & 4 & 5 & 6 & 7 \\\hline\hline
x_ix_{i+1} & q-2 & 2 & q-4 & q-1 & 1 & 3 & q-3 \\\hline
\end{array}\]
The leg $P_{2m+2}$ is labeled by using all even numbers in $[4, 2m+3]$ and all odd numbers in $[q-2m-4, q-6]=[2h+5]$ as follows:
\[\begin{array}{c|*{8}{|c}|}
j & 1 & 2 & 3 & 4 & 5 & \cdots & 2m & 2m+1 \\\hline\hline
y_iy_{i+1} & 4 & q-6 & 6 & q-8 & 8& \cdots &q-2m-4 & 2m+4 \\\hline
\end{array}\]
The leg $P_{2h+2}$ is labeled by using all odd numbers in $[5, 2h+3]\cup\{q\}$ and all even numbers in $[q-2h-3, q-5]=[2m+6, q-5]$ as follows:
\[\begin{array}{c|*{8}{|c}|}
i & 1 & 2 & 3 & 4 & 5 & \cdots & 2h & 2h+1 \\\hline\hline
z_iz_{i+1} & q & q-5 & 5 & q-7 & 7 & \cdots & q-2h-3 & 2h+3 \\\hline
\end{array}\]
Clearly, $f$ is a bijection. Now $f^+(u)=f^+(x_4)=f^+(z_2)=2q-5$, $f(x_6)=4=f^+(y_1)$, and $f^+(w)\in\{q, q-2\}$ for other vertex $w$. The theorem holds.
\end{proof}

\begin{theorem}\label{thm-sp(9,11,2m+1)} For $m\ge 2$, $\chi_{la}(Sp(9,11,2m+1))=4$. \end{theorem}

\begin{proof} Let $f$ be an edge labeling of $Sp(9,11,2m+1)$. Let $P_{10}=x_1\cdots x_{10}$, $P_{12} = y_1\cdots y_{12}$ and $P_{2m+2} = z_1\cdots z_{2m+2}$. Let $x_{10} = y_{12} = z_{2m+2}$. Now, $q = 2m+21$. We label $P_{10}$, $P_{12}$ and $P_{2m+2}$ as:

{\T{8}
\noindent$\begin{array}{c||*{9}{c|}}
i & 1 & 2 & 3 & 4 & 5 & 6 & 7 & 8 & 9  \\\hline\hline
x_ix_{i+1} & 2m+21 & 2m+14 & 6 & 2m+15 & 5 & m+9 & m+12 & m+8 & m+13 \\\hline
\end{array} \\[1mm]
\begin{array}{c||*{11}{c|}}
i & 1 & 2 & 3 & 4 & 5 & 6 & 7 & 8 & 9 & 10 & 11 \\\hline\hline 
y_iy_{i+1} & 2m+20 & 1 & 2m+19 & 2 & 2m+18 & 2m+17 & 3 & m+11 & m+10 & 4 & 2m+16\\\hline
\end{array} \\[1mm]
\begin{array}{c||*{11}{c|}}
i & 1 & 2 & 3 & 4 & 5 & 6 & 7 & 8 & 9 & 10 & 11 \\\hline\hline 
z_iz_{i+1} & m+14 & m+7 & 7 & 2m+13 & 8 & 2m+12 & \cdots & m+16 & m+5 & m+15 & m+6\\\hline
\end{array}$}\\

Clearly $f$ is a bijection. Now, $f^+(x_2) = f^+(y_6) = f^+(u) = 4m+35$ and $f^+(w) \in \{m+14,2m+20,2m+21\}$  for other vertex $w$. Thus, $f$ is a required local antimagic 4-coloring.
\end{proof}


\begin{theorem}\label{thm_sp(13,2m+1,2n+1)} For $m,n\ge 2$, $\chi_{la}(Sp(13,2m+1,2n+1))=4$. \end{theorem}

\begin{proof} Let $f$ be an edge labeling of $Sp(13,2m+1,2n+1)$. Let $P_{14} = x_1\cdots x_{14}$, $P_{2m+2}=y_1\cdots y_{2m+2}$ and $P_{2n+2} = z_1\cdots z_{2n+2}$. Let $x_{14} = y_{2m+2} = z_{2n+2}$. Now, $q = 2m+2n+15$. We label $P_{14}$, $P_{2m+2}$ and $P_{22+2}$ as:

{\T{8}
\noindent$\begin{array}{c||*{7}{c|}}
i & 1 & 2 & 3 & 4 & 5 & 6 & 7    \\\hline\hline
x_ix_{i+1} & 2m+2n+13 & 2 & 2m+2n+11 & 4 & 3 & 2m+2n+12 &  2m+2n+10  \\\hline
x_{i+7}x_{i+8} & 5 & 2m+2n+8 & 2m+2n+14 & 1 & 6 & 2m+2n+9 &\\\hline
\end{array}\\[1mm]
\begin{array}{c||*{8}{c|}}
i & 1 & 2 & 3 & 4 & 5 & 6 & 7 & 8 \\\hline\hline
y_iy_{i+1} & 2m+2n+15 & 2m+2n+7 & 8 & 2m+2n+5 & 10 & \cdots & 2n+9 & 2m+6 \\\hline
\end{array}\\[1mm]
\begin{array}{c||*{8}{c|}}
i & 1 & 2 & 3 & 4 & 5 & 6 & 7 & 8 \\\hline\hline
z_iz_{i+1} & 7 & 2m+2n+6 & 9 & 2m+2n+4 & \cdots & 2n+5 & 2m+8 & 2n+7  \\\hline
\end{array}$}\\

Clearly $f$ is a bijection. Now, $f^+(x_7) = f^+(u) = 4m+4n+22$ and $f^+(w) \in \{7,2m+2n+13,2m+2n+15\}$  for other vertex $w$. Thus, $f$ is a required local antimagic 4-coloring.
\end{proof}

\begin{example} $Sp(13,9,15)$\\
$\begin{array}{*{15}{c}}
35 & 2 & 33 & 4 & 3 & 34 & 32 & 5 & 30 & 36 & 1 & 6 & 31\\
37 & 29 & 8 & 27 & 10 & 25 & 12 & 23 & 14\\
7 & 28 & 9 & 26 & 11 & 24 & 13 & 22 & 15 & 20 & 17 & 18 & 19  & 16 & 21\\
\end{array}$\\
$(7,35,37,66)$
\end{example}


\begin{conjecture} For $d\ge 3$, $y_1, y_2,\ldots, y_d\ge 2$ and $d(d+1) \le 2(2q-1)$, $\chi_{la}(Sp(y_1,y_2,\ldots,y_d))=d+1$ except $Sp(2^{[n]},3^{[m]})$ for $(n,m)\in\{(4,0), (5,0), (6,0), (0,10), (1,8), (1,9), (2,7), (2,8), (3,5), (3,6), (4,4), (4,5), (5,3)\}$. \end{conjecture} 

\section{Appendix}

Following are local antimagic $(n+m+1)$-labelings of the spider $Sp(2^{[n]},3^{[m]})$, for some $(n,m)$. The numbers listed in the parenthesis are the induced vertex colors.

\small

\nt\begin{minipage}[b]{8.5cm}
$(n,m)=(0,3)$:\\\\
$\begin{array}{ccc}
8 & 1 & 7\\
6 & 2 & 4\\
9 & 5 & 3\end{array}$\\
(14, 9, 8, 6)
\end{minipage}
\begin{minipage}[b]{8.5cm}
$(n,m)=(0,4)$:\\
$\begin{array}{ccc}
12 & 5 & 4\\
11 & 6 & 3\\
10 & 7 & 2\\
9 & 8 & 1\end{array}$\\
(17, 12, 11, 10, 9)
\end{minipage}\\

\nt\begin{minipage}[b]{8.5cm}
$(n,m)=(0,5)$:\\\\
$\begin{array}{ccc}
15 & 6 & 5\\
14 & 7 & 4\\
13 & 8 & 3\\
12 & 9 & 2\\
11 & 10 & 1\end{array}$\\
(21, 15, 14, 13, 12, 11)
\end{minipage}
\begin{minipage}[b]{8.5cm}
$(n,m)=(0,6)$:\\
$\begin{array}{ccc}
18 & 8 & 7\\
17 & 9 & 6\\
16 & 10 & 5\\
15 & 11 & 4\\
14 & 12 & 3\\
2 & 13 & 1\end{array}$\\
(26, 18, 17, 16, 15, 14, 2)
\end{minipage}\\

\nt\begin{minipage}[b]{8.5cm}
$(n,m)=(0,7)$:\\
$\begin{array}{ccc}
21 & 9 & 8\\
20 & 10 & 7\\
19 & 11 & 5\\
18 & 12 & 4\\
17 & 13 & 3\\
16 & 14 & 2\\
6 & 15 & 1
\end{array}$\\
(30, 21, 20, 19, 18, 17, 16, 6)
\end{minipage}
\begin{minipage}[b]{8.5cm}
$(n,m)=(0,8)$:\\
$\begin{array}{ccc}
24 & 13 & 6\\
23 & 14 & 5\\
22 & 15 & 4\\
21 & 16 & 3\\
20 & 17 & 2\\
19 & 18 & 1\\
10 & 12 & 7\\
8 & 11 & 9
\end{array}$\\
(37, 24, 23, 22, 21, 20, 19, 10, 8)
\end{minipage}\\

\nt\begin{minipage}[b]{8.5cm}
$(n,m)=(0,9)$:\\
$\begin{array}{ccc}
27 & 18 & 9\\
26 & 19 & 8\\
25 & 20 & 7\\
24 & 21 & 6\\
23 & 22 & 5\\
17 & 10 & 4\\
16 & 11 & 3\\
15 & 12 & 2\\
14 & 13 & 1
\end{array}$\\
(45, 27, 26, 25, 24, 23, 17, 16, 15, 14)
\end{minipage}\\

\nt\begin{minipage}[b]{8.5cm}
$(n,m)=(1,2)$:\\
$\begin{array}{ccc}
6 & 2\\
7 & 1 & 5\\
8 & 3 & 4\end{array}$\\
(11, 8, 7, 6)
\end{minipage}
\nt\begin{minipage}[b]{8.5cm}
$(n,m)=(1,3)$:\\
$\begin{array}{ccc}
5 & 4\\
11 & 6 & 3\\
10 & 7 & 2\\
9 & 8 & 1\end{array}$\\
(17, 11, 10, 9, 5)
\end{minipage}
\\

\nt
\begin{minipage}[b]{8.5cm}
$(n,m)=(1,4)$:\\
$\begin{array}{ccc}
10 & 1\\
14 & 5 & 9\\
13 & 6 & 4\\
12 & 7 & 3\\
11 & 8 & 2
\end{array}$\\
(19, 14, 13, 12, 11, 10)
\end{minipage}
\begin{minipage}[b]{8.5cm}
$(n,m)=(1,5)$:\\
$\begin{array}{ccc}
8 & 7\\
17 & 9 & 6\\
16 & 10 & 5\\
15 & 11 & 4\\
14 & 12 & 3\\
2 & 13 & 1\end{array}$\\
(26, 17, 16, 15, 14, 8, 2)
\end{minipage}\\

\nt \begin{minipage}[b]{8.5cm}
$(n,m)=(1,6)$:\\
$\begin{array}{ccc}
7 & 8\\
20 & 9 & 6\\
19 & 10 & 5\\
18 & 11 & 4\\
17 & 12 & 3\\
16 & 13 & 2\\
15 & 14 & 1
\end{array}$\\
(29, 20, 19, 18, 17, 16 15, 7)
\end{minipage}
\begin{minipage}[b]{8.5cm}
$(n,m)=(1,7)$:\\
$\begin{array}{ccc}
13 & 6\\
23 & 14 & 5\\
22 & 15 & 4\\
21 & 16 & 3\\
20 & 17 & 2\\
19 & 18 & 1\\
10 & 12 & 7\\
8 & 11 & 9
\end{array}$\\
(37, 23, 22, 21, 20, 19, 13, 10, 8)
\end{minipage}
\begin{minipage}[b]{8.5cm}

\end{minipage}\\

\nt \begin{minipage}[b]{8.5cm}
$(n,m)=(2,1)$:\\
$\begin{array}{ccc}
6 & 1\\
5 & 2\\
7 & 4 & 3 \end{array}$\\
(11, 7, 6, 5)
\end{minipage}
\begin{minipage}[b]{8.5cm}
$(n,m)=(2,2)$:\\
$\begin{array}{ccc}
5 & 4\\
6 & 3\\
10 & 7 & 2\\
9 & 8 & 1\end{array}$\\
(17, 10, 9, 6, 5)
\end{minipage}\\

\nt \begin{minipage}[b]{8.5cm}
$(n,m)=(2,3)$:\\
$\begin{array}{ccc}
12 & 1\\
11 & 2\\
13 & 4 & 6\\
10 & 7 & 5\\
9 & 8 & 3\\
\end{array}$\\
(17, 13, 12, 11, 10 ,9)
\end{minipage}
\begin{minipage}[b]{8.5cm}
$(n,m)=(2,4)$:\\
$\begin{array}{ccc}
5 & 8\\
6 & 7\\
16 & 9 & 4\\
15 & 10 & 3\\
14 & 11 & 2\\
13 & 12 & 1
\end{array}$\\
(25, 16, 15, 14, 13, 6, 5)
\end{minipage}\\

\nt
\begin{minipage}[b]{8.5cm}
$(n,m)=(2,5)$:\\
$\begin{array}{ccc}
7 & 8\\
9 & 6\\
19 & 10 & 5\\
18 & 11 & 4\\
17 & 12 & 3\\
16 & 13 & 2\\
15 & 14 & 1
\end{array}$\\
(29, 19, 18, 17, 16 15, 9, 7)
\end{minipage} \begin{minipage}[b]{8.5cm}
$(n,m)=(2,6)$:\\
$\begin{array}{ccc}
13 & 6\\
14 & 5\\
22 & 15 & 4\\
21 & 16 & 3\\
20 & 17 & 2\\
19 & 18 & 1\\
10 & 12 & 7\\
8 & 11 & 9
\end{array}$\\
(37, 22, 21, 20, 19, 14, 13, 10, 8)
\end{minipage}\\

\nt \begin{minipage}[b]{8.5cm}
$(n,m)=(3,1)$:\\
$\begin{array}{ccc}
8 & 1\\
2 & 7\\
5 & 4\\
9 & 6 & 3
\end{array}$\\
(15, 9, 8, 5, 2)
\end{minipage}
\begin{minipage}[b]{8.5cm}
$(n,m)=(3,2)$:\\
$\begin{array}{ccc}
3 & 8\\
5 & 6\\
7 & 4\\
12 & 9 & 2\\
11 & 10 & 1
\end{array}$\\
(21, 12, 11, 7, 5, 3)
\end{minipage}\\

\nt \begin{minipage}[b]{8.5cm}
$(n,m)=(3,3)$:\\
$\begin{array}{ccc}
5 & 8\\
6 & 7\\
9 & 4\\
15 & 10 & 3\\
14 & 11 & 2\\
13 & 12 & 1
\end{array}$\\
(25, 15, 14, 13, 9, 6, 5)
\end{minipage}
\begin{minipage}[b]{8.5cm}
$(n,m)=(3,4)$:\\
$\begin{array}{ccc}
7 & 8\\
9 & 6\\
10 & 5\\
18 & 11 & 4\\
17 & 12 & 3\\
16 & 13 & 2\\
15 & 14 & 1
\end{array}$\\
(29, 18, 17, 16 15, 10, 9, 7)
\end{minipage}\\

\nt \begin{minipage}[b]{8.5cm}
$(n,m)=(4,1)$:\\
$\begin{array}{ccc}
3 & 8\\
5 & 6\\
7 & 4\\
9 & 2\\
11 & 10 & 1
\end{array}$\\
(21, 11, 9, 7, 5, 3)
\end{minipage}
\begin{minipage}[b]{8.5cm}
$(n,m)=(4,2)$:\\
$\begin{array}{ccc}
5 & 8\\
6 & 7\\
9 & 4\\
10 & 3\\
14 & 11 & 2\\
13 & 12 & 1
\end{array}$\\
(25, 14, 13, 10, 9, 6, 5)
\end{minipage}\\

\nt \begin{minipage}[b]{8.5cm}
$(n,m)=(4,3)$:\\
$\begin{array}{ccc}
7 & 8\\
9 & 6\\
10 & 5\\
11 & 4\\
17 & 12 & 3\\
16 & 13 & 2\\
15 & 14 & 1
\end{array}$\\
(29, 17, 16, 15, 11, 10, 9, 7)
\end{minipage}
\begin{minipage}[b]{8.5cm}
$(n,m)=(5,1)$:\\
$\begin{array}{ccc}
5 & 8\\
6 & 7\\
9 & 4\\
10 & 3\\
11 & 2\\
13 & 12 & 1\end{array}$\\
(25, 13, 11, 10, 9, 6, 5)
\end{minipage}\\

\nt \begin{minipage}[b]{8.5cm}
$(n,m)=(5,2)$:\\
$\begin{array}{ccc}
7 & 8\\
9 & 6\\
10 & 5\\
11 & 4\\
12 & 3\\
16 & 13 & 2\\
15 & 14 & 1
\end{array}$\\
(29, 16, 15, 12, 11, 10, 9, 7)
\end{minipage}
 \begin{minipage}[b]{8.5cm}
$(n,m)=(6,1)$:\\
$\begin{array}{ccc}
7 & 8\\
9 & 6\\
10 & 5\\
11 & 4\\
12 & 3\\
13 & 2\\
15 & 14 & 1
\end{array}$\\
(29, 15, 13, 12, 11, 10, 9, 7)
\end{minipage}


\begin{thebibliography}{99}

\bibitem{Arumugam-Wang} S. Arumugam, Y.C. Lee, K. Premalatha, T.M. Wang, On local antimagic vertex coloring for corona products of graphs, (2018) arXiv:1808.04956v1.

\bibitem{Arumugam} Arumugam S., Premalatha K., Bac\v{a} M., Semani\v{c}ov\'{a}-Fe\v{n}ov\v{c}\'{i}kov\'{a}, A.: Local antimagic vertex coloring of a graph. {\it Graphs and Combin.} {\bf33}, 275--285 (2017) .




\bibitem{LNS-GC} Lau G.C., Ng H.K., Shiu W.C.: Affirmative solutions on local antimagic chromatic number. Graphs and Combinatorics, (online 2020). https://doi.org/10.1007/s00373-020-02197-2

\bibitem{LauNgShiu-DMGT} Lau G.C., Shiu W.C., Ng H.K., : On local antimagic chromatic number of cycle-related join graphs. Discuss. Math. Graph Theory, (online 2018). https://doi:10.7151/dmgt.2177.

\bibitem{LSN} Lau G.C., Shiu W.C., Ng H.K.: On local antimagic chromatic number of graphs with cut-vertices, (2020), arXiv:1805.04801.


\bibitem{LauShiuNg-pendants} Lau G.C., Shiu W.C., Ng H.K.: On number of pendants in local antimagic chromatic number, submitted (2020).


\end{thebibliography}
\end{document}